\documentclass[11pt]{amsart}
\usepackage{amssymb}
\usepackage{wasysym}
\usepackage{tikz}
\usepackage{enumitem}
\usepackage{fontawesome}
\usepackage{amsmath}
\usepackage{xr}

\newenvironment{Exercise}%
{\smallskip \noindent {\sc An exercise?}}%
{\bigskip}
\newenvironment{Convention}%
{\smallskip \noindent {\sc A convention.}}%
{\bigskip}

\usepackage{etoolbox}

\newenvironment{Claim}%
{\smallskip \noindent {\sc Claim.}}%
{\bigskip}

\usepackage[hidelinks]{hyperref}

\newfont{\wcyr}{wncyr10 at 10pt}
\newfont{\wcyi}{wncyi10 at 10pt}
\newfont{\swcyi}{wncyi10 at 8pt}
\newfont{\wcyb}{wncyb10 at 10pt}

\newenvironment{Corollary}%
   {\smallskip \noindent {\sc Corollary.}}%
   {\bigskip}
\newenvironment{Definition}%
   {\smallskip \noindent {\sc Definition.}}%
   {}
   {\smallskip \noindent {\sc Example.}}%
   {\bigskip}
   {\smallskip \noindent {\sc Examples.}}%
   {\bigskip}
\newenvironment{Proof}%
   {\noindent {\sc Proof:}}%
   {\hfill\raisebox{-2pt}{\faThumbsOUp} \bigskip}
\newenvironment{Proposition}%
   {\smallskip \noindent {\sc Proposition.}}%
   {\bigskip}
\newenvironment{Theorem}%
   {\smallskip \noindent {\sc Theorem.}}%
   {\bigskip}
\newenvironment{Lemma}%
   {\smallskip \noindent {\sc Lemma.}}%
   {\bigskip}
\newenvironment{Remark}%
   {\smallskip \noindent {\sc Remark.}}%
   {\bigskip}
   {\smallskip \noindent {\sc Remarks.}}%
   {\bigskip}
   {\smallskip \noindent {\sc Question.}}%
   {\bigskip}
   {\smallskip \noindent {\sc Questions.}}%
   {\bigskip}
   {\smallskip \noindent {\sc Conjecture.}}%
   {\bigskip}

\newcommand{\cB}{{\mathcal B}}
\newcommand{\bbC}{\mathbb{C}}

\newcommand{\bbD}{\mathbb{D}}

\newcommand{\cH}{{\mathcal H}}

\newcommand{\N}{\mathbb{N}}
\newcommand{\Q}{\mathbb{Q}}
\newcommand{\R}{\mathbb{R}}

\newcommand{\cO}{{\mathcal O}}
\newcommand{\cP}{{\mathcal P}}

\newcommand{\cU}{{\mathcal U}}

\DeclareMathOperator{\diam}{diam}

\DeclareMathOperator{\Span}{span}

\newcommand{\cstar}{$\mathrm{C}^*$}
\newcommand{\cst}{\mathrm{C}^*}

\newcommand{\newpar}{\paragraph}

\newcommand{\bbN}{\N}

\newcommand{\bbQ}{\Q}
\newcommand{\bbR}{\mathbb R}

\DeclareMathOperator{\Sep}{Sep}
\DeclareMathOperator{\ex}{ex}
\DeclareMathOperator{\ap}{ap}

\newcommand{\sfT}{\mathsf T}
\newcommand{\sfS}{\mathsf S}
\newcommand{\sfP}{\mathsf P}
\newcommand{\sfA}{\mathsf A}
\newcommand{\sfB}{\mathsf B}

\newcommand{\sfX}{\mathsf X}
\newcommand{\sfY}{\mathsf Y}
\newcommand{\sfZ}{\mathsf Z}

\newcommand{\bfB}{\mathbf B}

\newcommand{\bbNN}{{\bbN}^{\bbN}}
\newcommand{\bbNlN}{{\bbN}^{<\bbN}}
\newcommand{\rs}{\restriction}

\DeclareMathOperator{\jsigma}{j\sigma}

\newcommand{\bSigma}{\boldsymbol\Sigma}
\newcommand{\bPi}{\boldsymbol\Pi}

\newenvironment{ProofOf}[1]%
{\noindent {\sc #1:}}%
{\hfill\raisebox{-2pt}{\faThumbsOUp} \bigskip}

{\smallskip \noindent {\sc Problem.}}%
{\bigskip}

\newcommand{\OCAT}{{OCA$_T$}}
\newcommand{\MA}{{MA}}

\newcommand{\CH}{{CH}}
\newcommand{\ZFC}{{ZFC}}
\newcommand{\ZF}{{ZF}}
\newcommand{\ZFCstar}{{ZFC*}}
\newcommand{\DC}{{DC}}

\title[Set-theoretic absoluteness for analysts]{Set-theoretic absoluteness for analysts and separable \cstar-algebras without Choice}
\author{Bruce Blackadar, Ilijas Farah} 


\thanks{I.F. is partially supported by NSERC}
\thanks{ChatGPT was used to `referee' the final version of this paper}

\address[B.B]{Department of Mathematics and Statistics,
University of Nevada, Reno\\
Reno, NV, USA  89557}
\email{bruceb@unr.edu}

\address[I.F]{Department of Mathematics and Statistics,
	York University,
	4700 Keele Street,
	Toronto, Ontario, Canada, M3J
	1P3\\
Matemati\v cki Institut SANU\\
	Kneza Mihaila 36\\
	11\,000 Beograd, p.p. 367\\
	Serbia}
\email{ifarah@yorku.ca}

\date{\today}

\begin{document}

\begin{abstract}
Basic theory of separable \cstar-algebras can be developed without the Axiom of Choice, and it does not depend on the Continuum Hypothesis, Martin's Axiom, and other standard set-theoretic assumptions. This can be proved in two ways. First, by showing that the standard proofs do not require Choice. Second, by utilizing set-theoretic absoluteness theorems. 
We provide an introduction to projective complexity and absoluteness for analysts. We also give some limiting examples consistent with \ZF, such as a commutative \cstar-algebra concretely represented  on a Hilbert space but not isomorphic to $C(X)$ for any compact Hausdorff space $X$ and whose state space is not compact and has no extreme points. 
\end{abstract}

\maketitle
		The goal of this paper is twofold. In addition to the results stated in the next paragraph, we present some classical results on absoluteness relevant to functional analysis  that are well known to logicians  but not nearly as well advertised as they should be. 

We show that the theory of separable \cstar-algebras can be developed in \ZF{} (that is, without using any Choice). This includes proving the Gelfand--Naimark representation theorems  as well as the Spectral Mapping  Theorem for polynomials and developing continuous functional calculus for commuting normal elements. 
Some of our proofs are modifications of the standard ones, obtained by avoiding the use of Choice. Some other proofs require new ideas in order to avoid the use of Choice. Yet another batch of proofs proceeds by using the set-theoretic Shoenfield Absoluteness Theorem. This result (well known to logicians but regrettably not as well advertised as it deserves) implies that statements about standard Borel spaces of low quantifier complexity (to be precise, $\bSigma^1_2$ statements; see \S\ref{S.Logical}) that are provable in \ZFC, or even \ZFC{} together with the Continuum Hypothesis (or Martin’s Axiom, \OCAT,\dots) are provable in ZF. 
One of the main objectives of this paper is to present these results in a convenient form that can be utilized by analysts not familiar with set theory. 

On a slightly different note, we also show that in the absence of Choice (more precisely, assuming the existence of a Russell set) there is a concretely representable unital commutative \cstar-algebra that is not isomorphic to $C(X)$ for any compact Hausdorff space $X$. Finally, from the model-theoretic point of view, while the property of having a tracial state is provably axiomatizable in \ZFC, it is not provably axiomatizable in \ \ZF+\DC.

\tableofcontents

\section{\cstar-Algebras}

\newpar{}
We take the usual definition of a \cstar-algebra \cite[II.1.1.1]{BlackadarOperator}, with just the caution that ``complete'' means ``$\sigma$-complete.''  As in \cite[Definition~1.0.1]{blackadar2023hilbert}, we say that a metric space
  $(X,\rho)$ is {\em $\sigma$-complete} if, whenever $(A_n)$ is a decreasing sequence of
nonempty closed  subsets of $X$ with $\diam(A_n)\to0$, then $\bigcap_n A_n$ is nonempty (it is then necessarily a singleton). In \cite{BrunnerSB} spaces with this property were called \emph{Cantor complete}.  For separable metric spaces $\sigma$-completeness is equivalent to 
usual Cauchy completeness \cite[XII.15.9.7]{BlackadarReal} (so ``complete'' is unambiguous in this case).

\newpar{}
The two standard examples of \cstar-algebras are:
\begin{enumerate}
	\item[(i)] The algebra  $\cB(\cH)$ of all bounded linear operators on a Hilbert space $\cH$, or more generally any norm-closed *-subalgebra thereof.
	(The fact that $\cB(\cH)$ is $\sigma$-complete and has an adjoint operation making it a \cstar-algebra requires a little proof, 	see  \cite[7.0.1]{blackadar2023hilbert}, \cite[7.0.3]{blackadar2023hilbert}.)
	\item[(ii)]  $C_0(X)$, the complex-valued continuous functions vanishing at infinity on a locally
	compact Hausdorff space $X$ (with complex conjugation as adjoint).
\end{enumerate}
The two Gelfand--Naimark theorems, two of the most basic facts in the subject, assert that these
examples are universal in the categories of all \cstar-algebras and all commutative \cstar-algebras, respectively.  However, without some Choice these results can fail, and thus we will
have to carefully reestablish all basic properties of \cstar-algebras.

\section{Banach Space Prerequisites}
We define a Banach space to be a normed space that is $\sigma$-complete in the norm metric
(Cauchy completeness suffices in the separable case). 
The following result can be proved in ZF; the result is probably well known, but we state it for future reference. 

\newpar{}\begin{Proposition}\label{P.HBetc}
	Suppose that $X$ is a separable Banach space (real or complex). Then the following results hold.

	\begin{enumerate}
		\item \label{1.BS.HBe} Hahn--Banach extension theorem: For every subspace $Y$ of $X$ and a bounded linear functional $\varphi$ on $Y$, there is a linear functional $\psi$ on $X$ that extends $\varphi$ and has the same norm as $\varphi$. 
		\item \label{2.BS.BA} Banach--Alaoglu theorem: The closed unit ball $\bfB(X^*)$ of the dual space $X^*$ is weak*-compact. 
		\item \label{3.BS.w*} The weak* topology on $\bfB(X^*)$ is compact, metrizable, and separable. 
		\item \label{4.BS.KM} Krein--Milman theorem: If $K$ is a weak*-compact and convex subset of $X^*$, then $K$ is the closed convex hull of its extreme boundary, $\partial K$. 
		\item \label{5.BS.HBs} Hahn--Banach separation theorem: If $A$ and $B$ are disjoint convex subsets of $X$ and $A$ is open, then there is $\varphi\in X^*$ such that $Re(\varphi(x))<\inf_{y\in B} Re(\varphi(y))$ for all $x\in A$.
		\item \label{6.BS.OM}  The Open Mapping Theorem, Closed Graph Theorem, and
		Uniform Boundedness Theorem hold for maps with domain $X$ (and separable codomain for the Closed Graph Theorem). 
	\end{enumerate}
\end{Proposition}

\begin{Proof}	Let $x_n$, for $n\in \bbN$, be an enumeration of a dense subset of the unit ball of $X$.

\eqref{1.BS.HBe} It suffices to prove this in the case when $X$ is a real Banach space. The standard proof, using Zorn's Lemma (e.g., \cite[Lemma 2.3.2]{pedersen2012analysis}), can be performed using a countable dense set instead, as follows.
 Then $Y_n=\Span(Y\cup \{x_j\mid j\leq n\})$ defines an increasing sequence of linear subspaces of $X$ with dense union. Let $Y_0=Y$ and $\psi_0=\varphi$.  The standard proof of the Hahn--Banach theorem shows that for every $n$ and linear functional $\psi_n$ on $Y_n$ there is a closed interval $[\alpha,\beta]$ in $\bbR$ such that  for any $\gamma\in [\alpha,\beta]$,  $\psi_{n+1}(y+tx_n)=\psi_n(y)+t\gamma$   extends $\psi_n$  to $Y_{n+1}$ and satisfies $\|\psi_n\|=\|\psi_{n+1}\|$. Use $\gamma=\alpha$ to define $\psi_{n+1}$.  This defines $\psi$ on $\bigcup_nY_n$, and $\psi$ extends
 to $X$ by continuity. 

	\eqref{2.BS.BA}	For the Banach--Alaoglu theorem, note that the evaluation map from $\bfB(X^*)$ into $\bbD^n$ (where $\bbD$ denotes the closed unit disk in $\bbC$) defined by $\varphi\mapsto (\varphi(x_n))$ is, by the standard proof, a homeomorphism onto its range. Since~$\bbD$ is homeomorphic to $[0,1]^2$, and since $[0,1]^\bbN$ is compact by Tychonoff's theorem (no choice needed, cf.\ \cite[XII.11.6.9]{BlackadarReal}), $\bbD^\bbN$ is compact. The range of the evaluation map is closed, hence compact. 
	
	\eqref{3.BS.w*}	The standard metric on $\bfB(X^*)$, 
	\[
	d(\varphi,\psi)=\sum_n 2^{-n}|\varphi(x_n)-\psi(x_n)|
	\]
	is as required, by the standard proof.  $\bfB(X^*)$ is second countable, hence separable by
	\cite[XII.15.9.8]{BlackadarReal}. 
	
	\eqref{4.BS.KM} and 	\eqref{5.BS.HBs}: While we don’t know ordinary proofs of these assertions that do not use the Axiom of Choice, we can 	prove that a proof exists. See  \ref{S.KreinMilman} and  \ref{S.HBseparation}.
		
	\eqref{6.BS.OM}:  The Baire Category Theorem for separable complete metric spaces
	requires no Choice.  The proofs of the Open Mapping Theorem and Uniform Boundedness 
	Theorem (second proof) as written in \cite{BlackadarReal} require no additional choice
	(in the proof in XVIII.8.2.5 the $x_n$ can be systematically chosen from a dense sequence),
	and the Closed Graph Theorem is a simple corollary of the Open Mapping Theorem.
	\end{Proof}

		If $X$ is just a separable normed space, then the completion $\bar X$ is also separable, 
	and $X^*=\bar X^*$, so (1)--(5) (not (6)!) hold for $X$.  But there is a subtlety: if $\bar X$ is
	separable, it is not obvious that $X$ is also separable (even if $X$ is an inner product space).
	In this case, (1)--(5) still hold for $X$.

The proof  of Proposition~\ref{P.HBetc} \eqref{1.BS.HBe} shows the following. 

\newpar{}\begin{Proposition}\label{P.WO.HB}
	Suppose that $X$ is a Banach space with a well-ordered dense subset. Then the Hahn--Banach extension theorem holds for $X$. \qed 
\end{Proposition}

Note also the following special case of \cite[XII.15.9.8]{BlackadarReal}:
\newpar{}
\begin{Proposition}\label{ClosedSubspaceProp}
Let $X$ be a separable Banach space.  Then every closed subset (e.g.\ closed subspace) of $X$
is separable.  In particular, every \cstar-subalgebra of a separable \cstar-algebra is separable.
\end{Proposition}

\subsection{Extensions}

\newpar{}
If $X$ is a Banach space, $Y$ a closed subspace, and both $Y$ and $X/Y$ are separable, is
$X$ necessarily separable?  Without Choice, the answer is probably no in general: a weak
cross-section property is needed (which is automatic under Countable AC):

\newpar{}
\begin{Proposition}\label{ExtensionProp}
	Let $X$ be a Banach space and $Y$ a closed subspace of~$X$.  Then $X$ is separable if and only
	if $Y$ is separable and there is a countable subset $C$ of~$X$ whose image in $X/Y$ (with the
	quotient norm) is dense.
\end{Proposition}

\begin{Proof}
	If $X$ is separable, then $Y$ is separable (\ref{ClosedSubspaceProp}), and a countable dense set
	in $X$ works for $C$.  Conversely, let $C$ be such a subset of $X$, and $D$ a countable dense
	subset of $Y$.  Then $C+D$ is countable.  Let $\pi:X\to X/Y$ be the quotient map.  Let $x\in X$
	and $\epsilon>0$.  Then there is a $c\in C$ with $\|\pi(x)-\pi(c)\|_{X/Y}=\|\pi(x-c)\|_{X/Y}<\epsilon$.
	From the definition of the quotient norm and density of $D$ in $Y$, there is a $d\in D$ with
	$$\|(x-c)-d\|_X=\|x-(c+d)\|_X<\epsilon$$
	and thus $C+D$ is dense in $X$.
\end{Proof}

\newpar{}
Thus for $X$ to be separable, one needs (in addition to separability of~$Y$) a countable dense subset
of $X/Y$ on which there is a cross section in $X$.  Note that there will always be such a dense
set if $X/Y$ is finite-dimensional (one only needs to take a lift of a basis; recall that finitely
many choices can be made without any form of Choice).

\newpar{}
Thus it is unclear (without Countable AC) whether if $A$ is a \cstar-algebra, $J$ a closed ideal, and $J$ and $A/J$ are 
separable, then $A$ is necessarily separable; it is true if $A/J$ is finite-dimensional.

\subsection{Existence of states}
\newpar{}\label{P.ExistenceOfStates}
Recall that a \emph{state} on a \cstar-algebra is a positive linear functional of norm 1.  Without the Hahn--Banach Theorem it is not automatic that there are enough states on a \cstar-algebra to determine the norm: for example, if all sets of reals have the Property of Baire (e.g., if the Axiom of Determinacy is assumed or in Solovay's model in which all sets of reals are Lebesgue measurable, \cite{So:Model}), neither the commutative \cstar-algebra $\ell^\infty(\N)/c_0(\N)$ nor the Calkin algebra have any states at all (see \cite[Corollary~8.4.3]{blackadar2023hilbert} and \cite[Exercise 12.6.6]{farah2019combinatorial}), respectively). 

\newpar{}\label{P.ExistenceOfStates.1}
Continuing \ref{P.ExistenceOfStates}, one may object that Solovay's construction uses an inaccessible cardinal,  and  moreover that by a celebrated result of Shelah (\cite[Theorem~5.1]{shelah1984can}) a model of \ZF{} in which all sets of reals are Lebesgue measurable (or even a model in which all $\bSigma^1_3$ sets are Lebesgue measurable, see Definition~\ref{Def.Sigma1n}) cannot be constructed without assuming the existence of an inaccessible cardinal in some inner model. This is a problem because the theory ZFC+`There exists an  inaccessible cardinal $\kappa$' implies that the rank-initial segment  $V_\kappa$ of the von Neumann universe $V$ is a model of \ZF{} and in particular that \ZF{} (as well as ZFC) is consistent. Since by G\"odel's theorem\ZFC{} cannot prove its consistency (unless it is inconsistent, in which case it also proves $0=1$), this theory is strictly stronger than ZFC. This is not a place to go into the (fascinating) subject of large cardinals and consistency strengths of extensions of ZFC.  See for example \cite{kanamori2008higher} or \cite{Jech:SetTheory} for more  information. 
Fortunately for us, in \cite[Corollary~7.17]{shelah1984can} it was proved that a model of  the theory  \ZF+\DC+``every set of reals has the Property of Baire' can be constructed starting from a model of ZFC. (Here `reals' stand for `any Polish space' and \DC{} is the axiom of Dependent Choices.)

\newpar{}\label{Par.states} Let us say that a \cstar-algebra $A$ has \emph{sufficiently many states} if for every $x\in A$ there is a state $\varphi$ such that  GNS representation associated with $\varphi$ satisfies $\|\pi_\varphi(x)\|=\|x\|$. By the \cstar-axiom, this is equivalent to requiring $\|\pi_\varphi(x)\|=\|x\|$ only for positive $x\in A$. 
The Axiom of Choice implies that every \cstar-algebra has sufficiently many states, but the previous paragraph shows that this is not the case if the Axiom of Choice fails.

\newpar{}
\begin{Lemma}\label{L.states}  Every separable \cstar-algebra, and more generally every \cstar-algebra with a well-ordered dense subset,  has sufficiently many states. 
\end{Lemma} 

\begin{Proof}  Fix a separable \cstar-algebra $A$. We may assume that $A$ is unital. Fix $x\in A$. 
	We will prove that there is a state $\varphi$ on $A$ such that $\varphi(x^*x)=\|x\|^2$. In particular, 
if $A$ is a unital \cstar-algebra and $x\in A$, then $x^*x\in A_+$ (Proposition~\ref{PropII313}), and it follows from functional
calculus that $\|\lambda 1+x^*x\|=\lambda+\|x^*x\|$ for any $\lambda\geq0$.  Thus, if $S$ is
the two-dimensional subspace of $A$ spanned by 1 and $x^*x$, the linear functional $\phi$ on $S$
defined by $\phi(1)=1$ and $\phi(x^*x)=\|x^*x\|=\|x\|^2$ has norm~1.

The Hahn--Banach Theorem for separable Banach spaces does not require any Choice (Proposition~\ref{P.HBetc}); in fact $\phi$ has an algorithmically constructed extension to $A$ of norm 1 (once a countable dense sequence is fixed).  Since a linear functional $\psi$ on $A$ is a state if and only if it has norm 1 and $\psi(1)=1$ \cite[II.6.2.5]{BlackadarOperator}, we conclude that a separable \cstar-algebra  has sufficiently many states. 
\end{Proof} 

\newpar{} In \ZF{} one cannot prove that every complete, totally bounded, metric space is compact (see \cite[561Y (c)]{Fr:MT5}, also \cite{farah2023choice} or \cite[XII.15.12.9]{BlackadarReal}). We therefore include a proof for the space of states $\sfS(A)$ below.

\newpar{}\begin{Lemma}\label{L.S(A)}
	Suppose that $A$ is a separable unital  \cstar-algebra. Then the state space $\sfS(A)$ with respect to the weak*-topology is a separable compact metrizable space. 
	The pure state space $\sfP(A)$ is a Borel (more precisely, $G_\delta$, i.e., $\Pi^0_2$ in modern descriptive-set theoretic notation, see \cite{Ke:Classical} or \cite[XV.3.4.1]{BlackadarReal}) subset of $\sfS(A)$. 
	
	If $A$ is commutative, unital,  and separable, then $X=\sfP(A)$ is a  compact metric space and $A\cong C(X)$. 
\end{Lemma}

\begin{Proof} The state space is nonempty by Lemma~\ref{L.states}. The standard proof that $\sfS(A)$ is a weak*-closed subset of the unit ball of the Banach space dual of $A$ requires no choice. By the Banach--Alaoglu theorem (Proposition~\ref{P.HBetc} \eqref{2.BS.BA}  and \eqref{3.BS.w*}), the latter is a separable compact metric space. 
	
	Finally, by the Krein--Milman theorem  for separable compact convex sets (a proof is provided in \ref{S.KreinMilman}), $\sfS(A)$  is the closed convex hull of the space of its extreme points, which is the space of pure states. 
	
	That $\sfP(A)$ is a $G_\delta$ is a standard result, see.\ e.g.\  \cite[XVIII.9.5.10]{BlackadarReal}.
	
	The proof in case when $A$ is not necessarily separable but has a well-ordered dense subset is analogous and uses Proposition~\ref{P.WO.HB} (although $\sfP(A)$ will not be Borel in general).  
	
	Finally, if $A$ is unital and commutative then the standard proof that the space of characters is closed in $\sfS(A)$ and that it coincides with $\sfP(A)$ applies.
\end{Proof}

\newpar{}
\begin{Proposition}\label{FaithfulStateThm}
A separable \cstar-algebra has a faithful state.
\end{Proposition}

\begin{Proof}
We may assume $A$ is unital.  Then $\sfS(A)$ is weak-* separable (\ref{L.S(A)}).
Let $(\phi_n)$ be a dense sequence in $\sfS(A)$, and set $\phi=\sum_{n=1}^\infty 2^{-n}\phi_n$.  Then $\phi$ is a state on $A$.  We claim that~$\phi$ is faithful.  Fix a positive $a\in A$ of norm 1.
Then there is an (algorithmically constructed) state $\psi$ on $A$ with $\psi(a)=1$.
Approximating $\psi$ by a $\phi_n$, there is an $n$ with $\phi_n(a)>1/2$, and we have
$\phi(a)\geq2^{-n}\phi_n(a)>2^{-n-1}>0$.
\end{Proof}

\subsection{Spectrum and continuous functional calculus}

\newpar{}
If $A$ is a unital   Banach algebra and $x\in A$, the {\em spectrum} of $x$ is
$$\sigma(x)=\sigma_A(x)=\{\lambda\in\bbC:x-\lambda 1\mbox{ is not invertible in }A\}.$$
Since the invertible elements in $A$ form an open set, $\sigma(x)$ is a closed subset of $\bbC$.

If $A$ is a \cstar-algebra and $x=x^*\in A$, then $\sigma(x)\subseteq\R$ \cite[II.1.6.2]{BlackadarOperator} (this proof requires no Choice; see \cite[4.3]{GoodearlNotes} for a more purely algebraic
proof which transparently requires no Choice).

We have the following theorem; see e.g.\ \cite{GoodearlNotes} for a proof not requiring any Choice.

\newpar{}
\begin{Theorem}\label{SpectrumThm}
	Let $A$ be a unital  complex Banach algebra, and $x\in A$.
	Then $\sigma(x)$ is nonempty and compact, and
	\[
	\max\{|\lambda|:\lambda\in\sigma(x)\}=\lim_{n\to\infty}\|x^n\|^{1/n}=\inf_n\|x^n\|^{1/n}\ .
	\]
\end{Theorem}

\newpar{}
\begin{Corollary}  
	Let $A$ be a unital \cstar-algebra, $x=x^*\in A$.  Then $\sigma(x)\subseteq[-\|x\|,\|x\|]$, and at least
	one of $\|x\|,-\|x\|$ is in $\sigma(x)$.  More generally, if $x$ is normal, then 
	$\|x\|=\max\{|\lambda|:\lambda\in\sigma(x)\}$.
\end{Corollary}

\begin{Proof}
	By the \cstar-axiom, $\|x^{2^k}\|^{2^{-k}}=\|x\|$ for any $k$. 
	So we have that $\lim_{n\to\infty} \|x^n\|^{1/n}$ exists and equals $\inf_n\|x^n\|^{1/n}$,  and submultiplicativity of the norm implies $\|x^n\|^{1/n} \leq \|x\|$ for all $n$, so  we have $\|x^n\|^{1/n}=\|x\|$ for all~$n$.
	Thus the sequence in the statement of Theorem~\ref{SpectrumThm} is constant. 
	\end{Proof}

Combining this with the Spectral Mapping Theorem for polynomials \cite[XVI.10.3.15]{BlackadarReal}, we obtain:

\newpar{}
\begin{Corollary}
	Let $A$ be a unital \cstar-algebra, $x$ a normal element of~$A$, and $p$ a polynomial with
	complex coefficients.  Then 
	\[
	\|p(x)\|=\max\{|p(\lambda)|:\lambda\in\sigma(x)\}.
	\]
\end{Corollary}
See \ref{FullFunctCalc} for a much better result, even for the case $n=1$.

\newpar{}\label{S.CFC}
The usual facts about continuous functional calculus of self-adjoint elements of a \cstar-algebra
hold without any Choice.  
%
No Choice is needed to prove the Stone--Weierstrass Theorem (cf.\ \cite[XVIII.10.2.16]{BlackadarReal}).
Combining this with the Spectral Mapping Theorem for polynomials, we obtain:

\newpar{}
\begin{Proposition}\label{SAFunctCalc}
	Let $A$ be a unital \cstar-algebra, and $x=x^*\in A$.  Then $p\mapsto p(x)$ ($p$ a complex polynomial)
	extends to an isometric
	isomorphism from $C(\sigma(x))$ onto $C^*(x,1)$, the uniformly closed subalgebra of $A$
	generated by $x$ and 1.
\end{Proposition}

\smallskip

If $A$ is a nonunital \cstar-algebra, we also get continuous functional calculus of self-adjoint elements
of $A$: if $x=x^*\in A$, then $C^*(x)\cong C_0(\sigma(x))$, the complex-valued continuous
functions on $\sigma(x)$ (computed in $A^\dag$) vanishing at~0.

\newpar{}
As usual, we say an element $x$ in a \cstar-algebra $A$ is {\em positive} ($x\geq0$) if $x=x^*$ and
$\sigma(x)\subseteq[0,\infty)$.  Denote by $A_+$ the set of positive elements of~$A$.

From continuous functional calculus we obtain most of \cite[II.3.1.2]{BlackadarOperator} (we
later get all of this result):

\newpar{}
\begin{Proposition}\label{PropII312}
	Let $A$ be a (unital) \cstar-algebra and $x\in A$.  Then
	\begin{enumerate}
		\item[(i)]  If $x\geq0$ and $-x\geq0$, then $x=0$.
		\item[(ii)]  If $x=x^*$, then $x^2\geq0$.
		\item[(iii)]  If $x\geq0$, then $\|x\|=\max\{\lambda:\lambda\in\sigma(x)\}$.
		\item[(iv)]  If $x=x^*$ and $\|x\|\leq2$, then $x\geq0$ if and only if $\|1-x\|\leq1$.
		\item[(v)]  If $x=x^*$ and $f\in C(\sigma(x))$, then $f(x)\geq0$ if and only if $f(t)\geq0$
		for all $t\in\sigma(x)$.
		\item[(vi)]  If $x=x^*$ there is a unique decomposition $x=x_+-x_-$ in $C^*(x)\subseteq A$
		with $x_+,x_-\in A_+$ and $x_+x_-=0$.  We have $x_+=f(x)$ and $x_-=g(x)$, where
		$f(t)=\max(t,0)$ and $g(t)=-\min(t,0)$.  Thus every element of $A$ is a linear combination
		of four elements of $A_+$.
		\item[(vii)]  Every positive element $x$ of $A$ has a unique (in $C^*(x)$) positive square root
		$x^{1/2}=f_{1/2}(x)$, where $f_{1/2}(t)=t^{1/2}$.  More generally, if $x\geq0$ and $\alpha$
		is a positive real number, there is a positive $x^\alpha=f_\alpha(x)\in C^*(x)_+$, where
		$f_\alpha(t)=t^\alpha$; these elements satisfy $x^\alpha x^\beta=x^{\alpha+\beta}$, 
		$(x^{\alpha})^\beta=x^{\alpha\beta}$, $x^1=x$,
		and $\alpha\mapsto x^\alpha$ is continuous.  If $x$ is invertible, $x^\alpha$ is also
		defined for $\alpha\leq0$, with the same properties.
	\end{enumerate}
\end{Proposition}

These parts of \cite[II.3.1.2]{BlackadarOperator} are what is required to make the proof of \cite[II.3.1.3]{BlackadarOperator} work (with no Choice needed):

\newpar{}
\begin{Proposition}\label{PropII313}
	Let $A$ be a \cstar-algebra.  Then
	\begin{enumerate}
		\item[(i)]  $A_+$ is a closed cone in $A$; in particular, if $x,y\in A_+$, then $x+y\in A_+$.
		\item[(ii)]  If $x\in A$, then $x^*x\geq0$.
	\end{enumerate}
\end{Proposition}

\smallskip

It then follows from \ref{PropII313} and \ref{PropII312} (vii) that an element $a\in A$ is in $A_+$
if and only if it is of the form $x^*x$ for some $x\in A$.

\bigskip

We also have the following technical fact.  This is usually proved using the Gelfand Representation
Theorem, which requires some Choice (see Proposition~\ref{P.Russell.example} below for an example of a unital, commutative \cstar-algebra with empty Gelfand spectrum).  We give
a proof not requiring any Choice.

\newpar{}
\begin{Proposition}\label{UniquePosSqrtProp}
	Let $A$ be a \cstar-algebra, $0\leq x\in A$.  Then there is a unique positive 
	$a\in A$ with $a^2=x$.
\end{Proposition}

\begin{Proof}
	There is a canonical $a=x^{1/2}\in C^*(x)\subseteq A$ with $a^2=x$ (Proposition~\ref{PropII312} (vii)).  It remains to prove uniqueness of the square root in $A$. 
	Since $a$ is a uniform
	limit of polynomials in $x$, it commutes with any element of $A$ which commutes with $x$.
	If $b$ is any positive element of $A$ with $b^2=x$, then $b$ commutes with $x$ and hence
	with $a$.  Thus, replacing $A$ with $C^*(a,b)$, we may assume $A$ is commutative.  Since
	$$0=a^2-b^2=(a-b)(a+b)$$
	we have $a-b\perp a+b$, and thus $a-b\perp(a+b)^{1/2^n}$ for any $n\in\N$.  But $a$ and $b$, hence $a-b$,
	are in the hereditary \cstar-subalgebra of $A$ generated by $a+b$ (which is positive), and $\{(a+b)^{1/2^n}:n\in\N\}$ is an
	approximate unit for this hereditary \cstar-subalgebra, so we have
	$$0=(a-b)(a+b)^{1/2^n}\to a-b$$
	and hence $a-b=0$.
\end{Proof}

\newpar{}
We can then get uniqueness (in $A$) of the elements $x^\alpha$ of Proposition~\ref{PropII312}(vii) for all 
$\alpha$ (i.e.\ if $x\in A_+$ there is a unique $b\in A_+$ with~$b^{1/\alpha}=x$):
uniqueness follows from \ref{UniquePosSqrtProp} if $\alpha$ is a dyadic rational, and then for
all $\alpha$ by continuity, as follows. If in addition $c^{1/\alpha}=x$, if $\beta$ is close to 1 and $\alpha\beta$
is dyadic rational, then $(c^\beta)^{1/\alpha\beta}=x=(b^\beta)^{1/\alpha\beta}$, so $c^\beta=b^\beta$.  Letting $\beta\to1$, we obtain $c=b$.

\newpar{}
So if $x\in A$, then Proposition~\ref{PropII312} (vi) implies that $x$ can be canonically written $x=a-b+i(c-d)$ where $a,b,c,d\in A_+$, with $a,b,c,d$
continuous functions of $x$.  Thus, if $A$ is separable and $(x_n)$ is a dense sequence,
write $x_n=a_n-b_n+i(c_n-d_n)$; then $\{a_n,b_n,c_n,d_n:n\in\N\}$ is dense in $A_+$, so
$A_+$ is also separable.  A quicker argument: $A_+$ is closed in $A$, hence separable by
\ref{ClosedSubspaceProp}.

\subsection{Continuous functional calculus for $n$-tuples of commuting normal operators}

\newpar{}\label{S.cfc.normal}
It is not obvious whether functional calculus extends to normal elements.  The special thing about a self-adjoint element $x$ is that if $p$ is a complex polynomial, then $[p(x)]^*=\bar p(x)$ is another polynomial
in $x$, so the closure in $A$ of the polynomials in $x$ is a *-algebra, which fails if $x$ is not self-adjoint. Nevertheless, we will use a metamathematical workaround to obtain continuous functional calculus for tuples of commuting normal elements (Proposition~\ref{FullFunctCalc}).

\newpar{} Suppose that $n\geq 1$ and $\bar a=(a_1,\dots, a_n)$ is an $n$-tuple commuting normal elements of a \cstar-algebra $A$. The \emph{joint spectrum} of $\bar a$ is 
\[
\jsigma(\bar a)=\jsigma_A(\bar a)=\{\bar \lambda\in\bbC^n:a_j-\lambda_j 1, \mbox{ for $j\leq n$, generate a proper ideal in }A\}. 
\]
The following is the Spectral Mapping Theorem for polynomials, again a theorem of ZF. 

\newpar{}\begin{Theorem} \label{T.SpectralMapping} Suppose that $n\geq 1$ and $a_1,\dots, a_n$ are commuting normal elements of a \cstar-algebra $A$. 
	Then $\jsigma(a_1,\dots, a_n)$ is nonempty.
	If $f$ is a complex *-polynomial in $2n$ variables, then 
	\begin{multline*}
	\sigma(f(a_1,\dots, a_n, a_1^*,\dots, a_n^*))\\
	=\{f(\lambda_1, \dots, \lambda_n, \bar\lambda_1, \dots, \bar\lambda_n)\mid (\lambda_1,\dots, \lambda_n)\in \jsigma(\bar a)\}. 
	\end{multline*}
\end{Theorem}
\newpar{}A proof of the case of a single normal operator of the Spectral Mapping Theorem that does not use the Axiom of Choice  is given in \cite[Lemma~2]{whitley1968spectral}.  It can be modified to the case of $n$ commuting normal operators by using induction. 
We are indebted to Chris Schafhauser for sharing this proof with us.

Our proof, given in \ref{Pf.Spectral}, uses Shoenfield's Absoluteness Theorem (Theorem~\ref{T.Shoenfield}) combined with a straightforward computation of the complexity of the assertion of Theorem~\ref{T.SpectralMapping}.

\section{Representability, Gelfand--Naimark theorems for separable \cstar-algebras}

We will show that both Gelfand--Naimark theorems for separable \cstar-algebras hold in ZF. We will also give an example showing that if a Russell set exists   then there exists a unital commutative \cstar-algebra that is not isomorphic to $C(X)$ for any compact Hausdorff space $X$.   
A Russell set gives a particularly  strong failure of the Axiom of Choice for countable sets. 
\newpar{}
\begin{Definition}\label{RepCAlg}
	A \cstar-algebra $A$ is {\em representable} if it is isomorphic to a closed *-algebra of bounded operators
	on a Hilbert space.
\end{Definition}

\smallskip

By the Gelfand--Naimark Theorem, which requires the Hahn--Banach Theorem, every \cstar-algebra is representable.  In fact, representability of a \cstar-algebra is equivalent to a weak form of the Hahn--Banach Theorem:

\newpar{}\label{P.Representable}
\begin{Proposition}
	A \cstar-algebra $A$ is representable if and only if $A$ has enough states to determine its norm, i.e.
	$$\|x\|^2=\sup\{\phi(x^*x):\phi\in\sfS(A)\}$$
	for every $x\in A$, where $\sfS(A)$ is the set of states on $A$.
\end{Proposition}

\begin{Proof}
	If $A$ is representable, then vector states determine the norm.  Conversely, if states determine the norm,
	a direct sum of GNS representations is faithful (isometric).
\end{Proof}

\newpar{} By \ref{P.HBetc}  we obtain representability for separable \cstar-algebras, and more, without any Choice.

\newpar{}
\begin{Theorem}\label{RepThm}
	Let $A$ be a separable (unital) \cstar-algebra, and $x\in A$.  Then there is a state $\psi$ on $A$ with
	$\psi(x^*x)=\|x\|^2$.  In particular, the states on $A$ determine the norm, so $A$ is representable.
\end{Theorem}

\begin{Proof} By Proposition~\ref{P.Representable} it suffices to prove that $A$ has sufficiently many states, and this was shown in Lemma~\ref{L.states}. 
\end{Proof}

In fact, using \ref{FaithfulStateThm} we can do better, since the GNS representation from a faithful
state is faithful:

\newpar{}
\begin{Theorem}\label{SepRepThm}
Every separable \cstar-algebra has a faithful representation as a concrete \cstar-algebra of operators
on a separable Hilbert space.
\end{Theorem}


\newpar{} For a commutative \cstar-algebra $A$ and $a\in A$ let $\hat a$ denote the evaluation function on $\sfP(A)$, $\hat a(\varphi)=\varphi(a)$. This is the Gelfand transform from $A$ into the \cstar-algebra $C_0(\sfP(A))$, and the standard proof that it is continuous does not require Choice (see e.g., \cite{Arv:Short}).

\newpar{}\begin{Theorem}
	If $A$ is a separable and  commutative \cstar-algebra, then the Gelfand transform 
	\[
	A\ni a\mapsto \hat a\in C_0(\sfP(A))
	\]
	is an isomorphism. In particular, every separable commutative \cstar-algebra is isomorphic to $C_0(X)$ for some locally compact metrizable space $X$.  If  $A$ is unital, then it is isomorphic to $C(X)$ for some compact metrizable space~$X$. 
\end{Theorem}

\begin{Proof}
	We first prove the assertion with the additional assumption that~$A$ is unital. By Lemma~\ref{L.states}, $A$ has sufficiently many states. Since the state space of $A$ is compact and metrizable,  the Krein--Milman theorem (\ref{S.KreinMilman}) implies that it is equal to the closed convex hull of the space of pure states of~$A$. Therefore the Gelfand transformation is isometric. The image of $A$ is a *-subalgebra of $C(\sfP(A))$ that contains all scalars and (by definition) separates the points. Since the Stone--Weierstrass theorem can be proved without the Axiom of Choice (\ref{S.CFC}), the map is an isomorphism as required. 
	
	If $A$ is not unital, consider its unitization $\tilde A$. Then $\sfP(A)$ is naturally identified with an open dense subspace of $\sfP(\tilde A)$ that separates the points in the image of $A$ under the Gelfand map, and $A$ is identified with $C_0(\sfP(A))$. The result follows from the already proved result for the unitization of $A$. 
\end{Proof}


\subsection{A concretely representable, commutative, unital \cstar-algebra  that is not isomorphic to $C(X)$ for any compact Hausdorff space~$X$}
In \cite[Corollary~8.4.3]{farah2023choice} it was proved that  the unital commutative  \cstar-algebra $\ell^\infty(\bbN)/c_0(\bbN)$ is, in some models of ZF, not representable on a Hilbert space and therefore not of the form $C(X)$ for a compact Hausdorff space~$X$. The title of this subsection describes the refinement of this result proved below. 
A \emph{Russell set} is a Dedekind--finite  set that can be written as a union of countably many disjoint two-element sets (see \cite[5.7]{blackadar2023hilbert}). 

\newpar{} 
\begin{Proposition} \label{P.Russell.example} Suppose that there exists a Russell set $X$. Then the following holds. 
	\begin{enumerate}
		\item \label{1.Russell.example} Some unital commutative \cstar-algebra $A$  is concretely representable on a Hilbert space but it is not isomorphic to $C(Z)$ for any compact Hausdorff space~$Z$. This algebra is an inductive limit of a sequence of finite-dimensional \cstar-algebras (but it is not separable).  
		\item \label{2.Russell.example} The state space of $A$ is not compact, and has no extreme points. 
		\end{enumerate}
\end{Proposition}

\begin{Proof}
		\eqref{1.Russell.example} Let $X=\bigsqcup_n X_n$ be a Russell set, where each $X_n$ is a two-element set.  We first define a universal \cstar-algebra $A_0$ given by generators and relations; the existence of such algebra does not require any Choice.  We will then define a state~$\tau$ on $A_0$, and the required algebra will be the image of $A_0$ under the GNS representation associated with  $\tau$. 
	
	Generators of $A_0$ are $p_x$, for $x\in X$. The relations require that each $p_x$ is a projection (that is, $p_x=p_x^*=p_x^2$), that for every $n$ we have $\sum_{x\in X_n} p_x=1$, and that $[p_x,p_y]=0$ for all $x$ and $y$. 
	Thus $p_x p_y=0$ if and only if $x$ and $y$ are distinct and belong to the same $X_n$. The space of linear combinations of products $\prod_{x\in F} p_x$, where $F\subset X$ intersects each $X_n$ in at most one point (every such transversal is finite because $X$ is a Russell set), is a dense *-subalgebra of $A_0$. 
	On such  monomials define a functional by $\tau(\prod_{x\in F} p_x)=2^{-|F|}$. 
	We claim that $\|\tau\|=\tau(1)=1$. To see this, note that $\tau(1)=1$ and therefore $\|\tau\|\geq 1$.  To see that $\|\tau\|\leq 1$, take a finite linear combination $a$ of monomials in projections $p_x$. For a large enough $n$, $a$ is of the form $a=\sum_{x\in \prod_{j\leq n} X_j} \lambda_x \prod_{j\leq n} p_{x(j)}$ with complex coefficients $\lambda_{x}$. Then 
	\[
	|\tau(a)|=\biggl|2^{-n}\sum_{x\in \prod_{j\leq n} X_j} \lambda _x\biggr|\leq \max_{x\in \prod_{j\leq n} X_j} |\lambda_x|=\|a\|,
	\] 
	and $\|\tau\|\leq 1$ follows. It follows that $\tau$ is a contraction on a dense subalgebra of $A_0$. Its continuous extension to $A_0$ (denoted $\tau$ for convenience) still satisfies $\|\tau\|=\tau(1)=1$, and is therefore a state. Since $A_0$ is commutative, it is tracial.

	Then $A=\pi_\tau[A_0]$ is a concrete commutative \cstar-algebra. Fix $x\in X$. 
Since $\tau(p_x)=1/2$, $\pi_\tau(p_x)$ is a nonzero projection. 
In the following, for the simplicity of notation we  will write $p_x$ for $\pi_\tau(p_x)$.
	
	We need to prove that the pure state space of $A$ is empty. Assume $\varphi$ is a pure state on $A$. Then for every $n$ we would have $\varphi(p_x)=1$ for some $x\in X_n$, otherwise $\varphi$ would be a convex combination of $\psi_x(a)=\varphi (p_x a)/\varphi(p_x)$ and $\psi_y(a)=\varphi (p_y a)/\varphi(p_y)$, where $X_n=\{x,y\}$. 
	Thus from $\varphi$ one could define a subset of $X$, $\{x\in X\mid \varphi(p_x)=1\}$ that intersects each $X_n$ in exactly one point; contradiction. 
	
	This implies that the Gelfand spectrum of $A$ is empty, and therefore $A$ is not isomorphic to $C(Z)$ for a compact Hausdorff space $Z$. 
	
For each $n\geq 1$, let $Y_n=\bigsqcup_{1\leq k\leq n}X_k$.  Then $Y_n$ is a finite set with $2n$
elements, and $X$ is the increasing union of the $Y_n$. 
Let $A_n$ be the $\pi_\tau$-image of  set of linear combinations of projections of the form $\prod_{x\in F} p_x$, where $F\subset Y_n$ intersects each $X_k$ in at most one point.  Then $A_n$ is a unital \cstar-subalgebra of $A$
isomorphic to ${\mathbb C}^{2^n}$,
$A_n\subseteq A_{n+1}$, and $\bigcup_nA_n$ is dense in $A$.  In fact, if $B_n$ is the \cstar-subalgebra
generated by $\{p_x:x\in X\setminus Y_n\}$, then $A$ is naturally isomorphic to 
$A_n\otimes B_n$.

\eqref{2.Russell.example} Let $\sfS(A)$ be the state space of  the \cstar-algebra $A$ as in \eqref{1.Russell.example}. $\sfS(A)$ can be described as an inverse limit of finite-dimensional
simplexes: there is an inverse
system
$$
\sfS(A)\to \cdots \to \sfS(A_{-n-1})\to \sfS(A_{-n})\to\cdots\to \sfS(A_{-2})\to \sfS(A_{-1})
$$
where the connecting maps are given by restriction.  $\sfS(A_n)$ is a $(2^n-1)$-simplex, and
the connecting maps send extreme points to extreme points.  The space $\sfS(A)$ can be
naturally put at the index $-\infty$, and the restriction maps from $\sfS(A)$ to $\sfS(A_n)$ separate points.  They are even surjective. This is because for every $\psi\in \sfS(A_n)$, the tensor product of $\psi$ and the restriction of $\tau$ to $B_n$ belongs to $\sfS(A)$.  

To see that $\sfS(A)$ is not compact (in the weak-* topology), let $K_n$ be the set of states of $A$ whose restriction to $A_n$
is a pure state. Then $K_n$ is closed since the pure state space of $A_n$ is closed, and
nonempty since the restriction map from $\sfS(A)$ to $\sfS(A_n)$ is surjective.  
So $(K_n)$ is a decreasing sequence of nonempty closed subsets of $\sfS(A)$, and any element of $\bigcap_n K_n$ would define a choice function for $X$. Therefore $\bigcap_nK_n=\emptyset$
and $\sfS(A)$ is not compact. 
\end{Proof}	

\newpar{} There is an obvious linear map $\theta$ from $A$ as in Proposition~\ref{P.Russell.example}  to the Banach space of bounded continuous
affine functions from $\sfS(A)$ to ${\mathbb C}$ with supremum norm.  Since $A$ is representable, $\theta$ is
isometric,  but it is not obvious that $\theta$ is surjective (the usual proof of surjectivity
for a unital commutative \cstar-algebra uses various forms of Choice).

\newpar{}	It is not difficult to see that the proof of Proposition~\ref{P.Russell.example} goes through if the assumption that each $X_n$ is a two-element set is relaxed to asserting that each $X_n$ is finite. It also works if the set $X$ is amorphous and it can be presented as $X=\bigsqcup_{z\in Z} X_z$,  where~$Z$ is strongly amorphous, and~$X_z$ are pairwise disjoint finite sets. A set with these properties was used in \cite[6.2.3]{blackadar2023hilbert} to produce a Hilbert space without an orthonormal basis. We don't know whether the space $L^2(A,\tau)$ with $A$ and $\tau$ as in the proof of Proposition~\ref{P.Russell.example} has an orthonormal basis, or whether it is isomorphic to a subspace of a Hilbert space with an orthonormal basis. 
	
	\section{Separability of $\cB(\cH)$}

	In the usual (Choice) case, there is a well-known theorem:
	
	\newpar{}
	\begin{Theorem}\label{BHNormSepThm}
		Let $\cH$ be a Hilbert space.  Then $\cB(\cH)$ is norm-separable if and only if $\cH$ is 
		finite-dimensional.
	\end{Theorem}
	
	\newpar{}
	The usual proof (or at least one standard proof) requires at least the Countable AC.  But it turns
	out that all use of Choice can be avoided, at the cost of some additional argument.  We outline
	the standard argument, and then adapt it to the non-choice setting; results of \cite{blackadar2023hilbert} yield that all the subspaces and projections used
	are constructible without any Choice..

	One direction is clear: if $\cH$ has finite dimension $n$, then $\cB(\cH)\cong M_n(\bbC)$ has an
	obvious explicit countable dense subset, which requires no Choice.  So we need only the converse.


	Under the Countable AC,  any subspace of a separable metric space is separable.  However, this cannot be
	proved without the Countable AC. For example, a set of Cohen reals in Cohen's symmetric model is a nonseparable subset of $\bbR$. Such set is an example of a Cohen-finite set whose existence is equivalent to some curious statements about Hilbert spaces (see e.g. \cite[Proposition~6.0.3]{blackadar2023hilbert}).  However, we need only a special version. 
		Let $(X,\rho)$ be a metric space, and $\epsilon>0$.  A subset $\{x_i:i\in I\}$ of $X$ is
		{\em $\epsilon$-discrete} if $\rho(x_i,x_j)\geq\epsilon$ for all $i,j\in I$, $i\neq j$. 

\newpar{}\label{epsilon-discrete}		Every  $\epsilon$-discrete
		subset of a separable metric space $X$ is countable.
		To see this, $\{x_i:i\in I\}$ be an $\epsilon$-discrete subset of $X$, and let $(y_n)$ be a dense sequence
		in $X$.  For each $i\in I$, let $n_i$ be the smallest $n$ such that $\rho(x_i,y_n)<\frac{\epsilon}{2}$.
		Then $i\mapsto n_i$ is an injective function from $I$ to $\N$, so $I$ is countable.

	\newpar{}
	\begin{Proposition}\label{OrthProjSeqProp}
		Let $\cH$ be a Hilbert space.  If $\cB(\cH)$ contains a sequence of mutually orthogonal
		nonzero projections, then $\cB(\cH)$ is not norm-separable.
	\end{Proposition}
	
	\begin{Proof}
		Let $(P_n)$ be a sequence of mutually orthogonal nonzero projections in $\cB(\cH)$, and let
		$\cH_n$ be the range of $P_n$.  Then $(\cH_n)$ is a sequence of mutually orthogonal closed
		subspaces of $\cH$.  For each $S\subseteq\N$, let $\cH_S$ be the closed span of
		$\{\cH_n:n\in S\}$, and $P_S$ the projection onto $\cH_S$.  Then $\{P_S:S\subseteq\N\}$
		is a 1-discrete subset of $\cB(\cH)$ of cardinality $2^{\aleph_0}$.  Since $2^{\aleph_0}>\aleph_0$,
		$\cB(\cH)$ cannot be separable by \ref{epsilon-discrete}.
	\end{Proof}

	Thus to finish the proof of Theorem \ref{BHNormSepThm}, it suffices to find a sequence of
	mutually orthogonal nonzero projections in $\cB(\cH)$ for $\cH$ infinite-dimensional.  This is
	easily done in the usual (Countable AC) case by taking an orthonormal sequence of vectors.
	However, in the absence of the Countable AC, an infinite-dimensional Hilbert space need not
	contain an orthonormal sequence, or even an infinite orthonormal set \cite[6.2.3]{blackadar2023hilbert}; $\cB(\cH)$
	need not even contain any sequence of mutually orthogonal nonzero projections \cite[8.3.2]{blackadar2023hilbert}.
	So to prove \ref{BHNormSepThm}, we proceed indirectly.
	
	\newpar{}
	Suppose $\cH$ is a Hilbert space and $\cB(\cH)$ is separable.  Let $(T_n)$ be a dense sequence
	in $\cB(\cH)$, and for each $n$ let $S_n=(T_n+T_n^*)/2$ be the real part of $T_n$.  Then
	the $S_n$ are self-adjoint, and they are dense in $\cB(\cH)_{sa}$.  If $P$ is a projection
	in $\cB(\cH)$ and $0<\epsilon<\frac{1}{2}$, there is an $S_n$ with $\|P-S_n\|<\epsilon$.
	The spectrum of $S_n$ is contained in $(-\epsilon,\epsilon)\cup(1-\epsilon,1+\epsilon)$
	and hence $\frac{1}{2}\notin\sigma(S_n)$.  Let $f:\R\to\R$ be the function which is 0 on
	$\left ( -\infty,\frac{1}{2}\right )$ and 1 on $\left ( \frac{1}{2},+\infty\right )$; then $f$ is
	continuous on $\sigma(S_n)$, so we may form $P_n=f(S_n)$ (no Choice is needed for this
	functional calculus (\ref{SAFunctCalc})).  Then $P_n$ is a projection, and we have 
	$$\|P-P_n\|\leq\|P-S_n\|+\|S_n-P_n\|<\epsilon+\epsilon=2\epsilon$$
	and thus the sequence $(P_n)$ (defined only for the $n$ for which $\frac{1}{2}\notin\sigma(S_n)$)
	of projections is dense among all projections of $\cB(\cH)$.
	
	\newpar{}
	If $\xi$ is a nonzero vector in $\cH$, let $P$ be the rank-one projection onto the span of $\xi$.
	For any $\epsilon$, $0<\epsilon<1$, there is an $n$ for which $\|P-P_n\|<\epsilon$, and since
	$\epsilon<1$, $P_n$ is also rank one.  Thus the rank-one projections in the sequence $(P_n)$
	are dense in the rank-one projections of $\cB(\cH)$.
	
	\newpar{}
	Reindexing, let $(P_n)$ be a sequence of rank-one projections in $\cB(\cH)$ whose closure contains
	all rank-one projections.  For each $n$ set 
	$$Q_n=\bigvee_{k=1}^n P_k\ .$$
	Then $Q_n$ is a finite-rank projection or rank $\leq n$, and $(Q_n)$ is an increasing sequence.
	The supremum of the $Q_n$ is $I$ since every vector in $\cH$ is in the closure of the union of
	the ranges of the $P_n$.  There are two possibilities:
	\begin{enumerate}
		\item[(1)]  The $Q_n$ stabilize, i.e.\ there is an $n_0$ such that $Q_n=Q_{n_0}$ for all $n>n_0$.
		Then $Q_{n_0}=I$ and $\cH$ is finite-dimensional.
		\item[(2)]  The $Q_n$ do not stabilize, i.e.\ by passing to a subsequence they strictly increase.
		In this case, the projections $Q_{n+1}-Q_n$ form a sequence of mutually orthogonal nonzero
		projections in $\cB(\cH)$, contradicting that $\cB(\cH)$ is separable (\ref{OrthProjSeqProp}).
	\end{enumerate}
	This completes the proof of Theorem \ref{BHNormSepThm} with no use of Choice.

	\newpar{}\label{OrthRankOneSeq}
	We actually showed a little more than necessary:  the projections $Q_{n+1}-Q_n$ have rank one,
	i.e.\ if $\cH$ is infinite-dimensional and $\cB(\cH)$ is separable, then it contains a sequence of
	mutually orthogonal rank-one projections whose ranges span $\cH$ (which then contradicts \ref{OrthProjSeqProp}).
	This, however, does not directly show that $\cH$ is separable without the Countable AC: to obtain
	an orthonormal sequence, one would have to pick a unit vector in the range for each $n$.
	In fact, without the Countable AC there is a nonseparable Hilbert space $\cH$ with such a
	sequence of mutually orthogonal rank-one projections \cite[6.2.3, 6.2.4]{blackadar2023hilbert}.  We do not know how to
	directly prove without Choice that if $\cB(\cH)$ is separable, then $\cH$ is separable; of course
	this follows indirectly from our proof that $\cH$ must be finite-dimensional.

	\newpar{} After the first author proved Theorem~\ref{BHNormSepThm}, the second author realised that he had already seen the method for effectively finding a dense set of projections in \cite{Kec:C*} and even used it in a proof that the $K_0$ group of a separable \cstar-algebra is Borel-computable (\cite[Lemma~3.13(4)]{FaToTo:Turbulence}). 
	\section{Absoluteness}
	
	\newpar{} In this section we introduce set-theoretic absoluteness, but only after we emphasize that no foul play is involved in the ongoing discussion. The following Lemma provides an algorithm for finding an `honest' proof of any statement proved using absoluteness. As the algorithm is wildly inefficient and yields a formal proof which is likely to be incomprehensibly complicated, we state this fact only for reference. 

	\newpar{} \begin{Lemma} If there is a proof that a statement $\varphi$ is provable in \ZF{} (or any other recursively axiomatizable theory $T$), then there is an algorithm for finding a proof of $\varphi$ in \ZF{} (or $T$). 
	\end{Lemma}
	
	\begin{Proof} Since the axioms of the theory $T$ can be recursively  enumerated, and since a proof in $T$ is a finite sequence of sentences that satisfies strict and clear rules, one can enumerate all (formal) proofs in $T$. Knowing that $\varphi$ is provable in $T$, one knows that a proof for $\varphi$ will eventually appear in this list. 
	\end{Proof} 	
	
\subsection{Standard Borel spaces}

A topological space is \emph{Polish} if it is separable and completely metrizable. 
\emph{Standard Borel space} 
is a pair $(\sfX,\Sigma)$ where $\sfX$ is a set and $\Sigma$ is 
a $\sigma$-algebra of subsets of~$\sfX$ with the property that $\sfX$ carries a Polish topology such that 
$\Sigma$ is the algebra of Borel subsets of $\sfX$ in this topology. 
 Standard Borel spaces appear both in operator algebras and in set theory (\cite[\S 3]{Arv:Invitation}, \cite[\S 12.B]{Ke:Classical}). 
If $(X,\Sigma)$ is standard and $Y\in \Sigma$, then $(Y,\Sigma\cap \cP(Y))$ is standard (\cite[\S 13.4]{Ke:Classical}).

The Cartesian product of two standard Borel spaces carries a natural product Borel structure. This is the Borel structure associated with the product topology of the underlying Polish spaces, and therefore a standard Borel space.

\newpar{} Lemma~\ref{L.S(A)} implies that if $A$ is a separable \cstar-algebra then the Borel structure associated to the weak*-topology on each one of the spaces $\sfS(A)$ and $\sfP(A)$ is standard (recall that a $G_\delta$ in a Polish space is a Polish space).

\subsection{The projective hierarchy}
The projective hierarchy of subsets
of a standard Borel space $\sfX$ is defined by recursion as follows.

\newpar{}\begin{Definition} \label{Def.Sigma1n} Let $(\sfX,\Sigma)$ be a standard Borel space. 
	A subset $\sfA$ of $\sfX$ is $\bSigma^1_1$ (or \emph{analytic})\footnote{This is equivalent to being an analytic 
		subset of $\sfX$ (i.e., a continuous image of a Borel subset of some Polish space) in some, or any, Polish topology on $\sfX$ (\cite[25.A]{Ke:Classical}, see also Theorem~\ref{T.Isomorphic} below).} if there are a standard Borel space ($\sfY,\Sigma')$  and a Borel set $\sfB\subseteq \sfX\times \sfY$ such that 
	\begin{equation}\label{eq.Sigma11}
	\sfA=\{x\in \sfX\mid (\exists y\in \sfY) (x,y)\in \sfB\}.
	\end{equation}
	A subset $\sfA$ of $\sfX$ is $\bPi^1_1$ (or \emph{coanalytic}) if there are a standard Borel space ($\sfY,\Sigma')$  and a Borel set $\sfB\subseteq \sfX\times \sfY$ such that 
	\[
	\sfA=\{x\in \sfX\mid (\forall y\in \sfY) (x,y)\in \sfB\}.
	\]
	Equivalently, $\bPi^1_1$ sets are complements of $\bSigma^1_1$ sets. 
	
	For $n\geq 1$, $\bSigma^1_{n+1}$ and $\bPi^1_{n+1}$ sets are defined by recursion on $n$ as follows. A subset $\sfA$ of $\sfX$ is $\bSigma^1_{n+1}$ if there are a standard Borel space ($\sfY,\Sigma')$  and a~$\bPi^1_n$ set $\sfB\subseteq \sfX\times \sfY$ such that 
	\[
	\sfA=\{x\in \sfX\mid (\exists y\in \sfY) (x,y)\in \sfB\}.
	\] 
	A subset $\sfA$ of $\sfX$ is $\bPi^1_{n+1}$ if there are a standard Borel space ($\sfY,\Sigma')$  and a~$\bSigma^1_n$ set $\sfB\subseteq \sfX\times \sfY$ such that 
	\[
	\sfA=\{x\in \sfX\mid (\forall y\in \sfY) (x,y)\in \sfB\}.
	\] 
	Equivalently, $\bPi^1_n$ sets are complements of $\bSigma^1_n$ sets for all $n$. 
\end{Definition} 

\newpar{}\begin{Remark}The following remark may be safely ignored by most readers, but there is a reason why we used boldface symbols $\bSigma$ and $\bPi$ in Definition~\ref{Def.Sigma1n}. The lightface notation $\Sigma^1_1$ is reserved for the class of sets $\sfA$ such that the Polish spaces $\sfX$ and $\sfY$, as well as the Borel set $\sfB$, are effective (see \cite[40.B]{Ke:Classical}).  The technical details of the definition of `effective' are irrelevant for us, but there are only countably many recursive closed subsets of a given recursively presented Polish space. This for example includes the Cantor space,  $\bbR^n$ and $\cO_n$ for all $n$ and all separable AF algebras whose $K_0$ group is presented by recursive relations, but not all separable AF algebras.  From here one proceeds to recursively define the lightface hierarchies $\Sigma^1_n$ and $\Pi^1_n$. Tools from effective descriptive set theory, and the structure of the poset of $\Sigma^1_1$ subsets of $\cP(\bbN)$ in particular, were famously used to prove the Glimm--Effros dichotomy (\cite{HaKeLo:GE}), but we will not need them here.   
%
\end{Remark}

\subsection{Logical complexity I. Projective sets} \label{S.Logical} By unraveling Definition~\ref{Def.Sigma1n} one sees that a subset $\sfA$ of a standard Borel space $(\sfX_1,\Sigma_1)$ is $\bSigma^1_{2n-1}$ for $n\geq 1$ if there are standard Borel spaces $(\sfX_j, \Sigma_j)$, for $2\leq j\leq 2n$, and a Borel set $\sfB\subseteq \prod_{j=1}^{2n} \sfX_j$ such that 
\[
\sfA=\{x_1\mid (\exists x_2\in \sfX_2)(\forall x_3\in \sfX_3)\dots (\exists x_{2n}\in \sfX_{2n})(x_1,\dots, x_{2n})\in \sfB\}, 
\]
and that $\sfA$ is $\bSigma^1_{2n}$ for some $n\geq 1$ if there are standard Borel spaces $(\sfX_j, \Sigma_j)$, for $2\leq j\leq 2n+1$, and a Borel set $\sfB\subseteq \prod_{j=1}^{2n+1} \sfX_j$ such that 
\[
\sfA=\{x_1\mid (\exists x_2\in \sfX_2)(\forall x_3\in \sfX_3)\dots (\forall x_{2n+1}\in \sfX_{2n+1})(x_1,\dots, x_{2n+1})\in \sfB\}. 
\]
Analogous formulas for $\bPi^1_n$ sets are easily derived from the above by taking complements and therefore omitted.

\newpar{}Two standard Borel spaces are \emph{Borel-isomorphic} if there is a bijection~$f$ between them such that both $f$ and its inverse are Borel-measurable. The following is Kuratowski’s theorem. 

\newpar{}\begin{Theorem}\label{T.Isomorphic}
	Two standard Borel spaces are Borel-isomorphic if and only if they have the same cardinality, which is either countable or $2^{\aleph_0}$. 
	Any Borel-isomorphism between standard Borel spaces sends $\bSigma^1_n$ sets to $\bSigma^1_n$ sets and $\bPi^1_n$ sets to $\bPi^1_n$ sets  for every~$n$. 
\end{Theorem}

\begin{Proof} 
	If $\sfX$ is a countable standard Borel space, then its Borel $\sigma$-algebra is equal to  $\cP(\sfX)$, thus every two countable standard Borel spaces of the same cardinality are Borel-isomorphic. The second part is in this case vacuous.
	
	So it suffices to prove that every uncountable standard Borel space~$\sfX$ is Borel-isomorphic to the Cantor space~$K$. By Cantor--Bendixson analysis,~$\sfX$ contains a perfect subset (\cite[Theorem~6.4]{Ke:Classical}), hence there is a continuous injection from $K$ into $\sfX$. On the other hand, there is a Borel injection from $\sfX$ into $K$ (for an elegant proof of this fact due to Jenna Zomback see \cite[Lemma B.2.3]{farah2019combinatorial}). By the Borel version of the Schr\"oder--Bernstein theorem (\cite[Theorem~15.7]{Ke:Classical}), $\sfX$ is Borel-isomorphic to $K$. 
	
	
	It remains to prove that a Borel isomorphism between Polish spaces preserves the projective hierarchy. 
	Clearly, a Borel isomorphism sends Borel sets to Borel sets.
	Suppose that $f\colon \sfX\to \sfY$ is a Borel-measurable map between Polish spaces and that $\sfA\subseteq \sfX$ is a $\bSigma^1_1$ set. 
	Fix a standard Borel space~$\sfZ$ and  a Borel $\sfB\subseteq \sfX\times \sfZ$ such that $\sfA$ is the projection of $\sfB$. The graph of $f$ is a Borel subset of $\sfX\times \sfY$. The following subset of $\sfX\times \sfZ\times \sfY$: 
	\begin{equation}\label{eq.BorelIsomorphism}
		\{(b,a,f(a)): (b,a)\in \sfB\}
	\end{equation}
	is  an intersection of two Borel sets and therefore Borel, and its  projection to $\sfY$ is equal to $f[\sfA]$. 
	This implies that~$\bSigma^1_1$ sets are preserved by Borel-isomorphisms, and the statement for~$\bPi^1_1$ sets follows by taking complements. 
	
	Suppose that  $n$ is such that every Borel isomorphism between standard Borel spaces sends $\bSigma^1_n$ sets to $\bSigma^1_n$ sets and $\bPi^1_n$ sets to $\bPi^1_n$ sets. Fix a Borel isomorphism between standard Borel spaces $f\colon \sfX\to \sfY$. The proof that the image of a $\bPi^1_n$ set $\sfA\subseteq \sfX$ under a Borel map $f$ is $\bSigma^1_{n+1}$ is analogous to the proof in case when $n=0$: Fix a $\bPi^1_{n}$ set $\bfB\subseteq \sfX\times \sfZ$, note that the set in \eqref{eq.BorelIsomorphism} is an intersection of a $\bPi^1_n$ set and a Borel set, and therefore $\bPi^1_n$. Its projection to $\sfY$ is equal to $f[\sfA]$. 
	By taking complements, this conclusion extends to~$\bPi^1_{n+1}$ sets and completes the proof of the inductive step. 
\end{Proof} 

\newpar{}\begin{Exercise} \label{Ex.Luzin} As an illustration, we state what may appear as an innocent--looking `exercise' in real analysis, but  is much more than that. 
Consider a subset of $\bbR$ defined as follows. 
Let $A$ be a Borel subset of $\bbR^3$. Let $B$ be the projection of $A$ to the $xy$-plane and let $C$ be the complement of $B$. Let $D$ be the projection of $C$ to the $x$-axis. 

Is every set $D$ obtained in this manner Lebesgue measurable?

This question is asking whether every $\bSigma^1_2$ subset of $\bbR$ is Lebesgue measurable. The answer to this question is nothing short of fascinating. In G\" odel's constructible universe the answer is negative, and there is a well-ordering of the reals that is $\bSigma^1_2$ (\cite[Theorem~25.26]{Jech:SetTheory}). 
On the other hand, the existence of a measurable cardinal implies that the answer is positive (this follows from e.g. \cite[Theorem 14.1]{kanamori2008higher}). 
As if this was not fascinating enough,  if every $\bSigma^1_3$ set is Lebesgue measurable, then there is an inner model of \ZFC{} with an inaccessible cardinal (\cite{shelah1984can}).  
Thus the theory  \ZF+`all sets of reals are Lebesgue measurable' is equiconsistent with ZFC+`there exists an inaccessible cardinal' and of  consistency strength higher than that of   ZFC.  
\end{Exercise}

\subsection{Logical complexity II. Projective formulas}
\newpar{} We will need to define $\bSigma^1_n$ and $\bPi^1_n$ formulas. For completeness, let us first recall the definition of $\Sigma_n$ and $\Pi_n$ formulas. 
A formula in a first-order language is called $\Sigma_1$ if it has the form
\begin{equation}\label{eq.Sigma1}
(\exists x_1)(\exists x_2)\dots (\exists x_k)\varphi
\end{equation}
where $\varphi$ is quantifier-free. A negation of a $\Sigma_1$ formula is called $\Pi_1$. If $\varphi$ is $\Pi_n$ then  a formula as in \eqref{eq.Sigma1} is $\Sigma_{n+1}$, and negations of $\Sigma_{n+1}$ formulas are called $\Pi_{n+1}$. 

\newpar{} The definition involves  
\emph{second-order arithmetic},  a subsystem of \ZFC{} introduced by Hilbert and Bernays (see \cite[\S  12]{kanamori2008higher} or \cite{simpson2009subsystems}).  This is the  two-sorted structure 
\[
(\bbN,\bbNN,+,\cdot,\ex,<,0,1,\ap).
\]
The standard interpretation of this structure has two disjoint sorts, corresponding to the set of natural numbers  $\bbN$ and the Baire space
$\bbNN$. 
The functions $+,\cdot$, relation $<$, and constants $0,1$ have their usual interpretations on  $\bbN$.  
The function $\ex$ from $\bbN^2$ to $\bbN$ is interpreted as $\ex(m,n):=m^n$. (`$\ex$' stands for `exponential.') 
The only function involving the sort $\bbNN$ is $\ap$, 
interpreted as $\ap(a,m):=a(m)$. (`$\ap$' stands for `application.')

\newpar{} The language of second-order arithmetic 
has two sorts of variables. The variables ranging over $\bbN$ are $x^0_j$, for $j\in \bbN$
and the variables ranging over $\bbNN$ are $x^1_j$, for $j\in \bbN$.       
Correspondingly we have two sorts of quantifiers, the \emph{number quantifiers} $\exists^0, \forall^0$
and the \emph{function quantifiers} $\exists^1, \forall^1$. 
\emph{Bounded quantifiers} are  of the form 
$(\exists^0 x\leq y)$ or 
$(\forall^0 x\leq y)$.\footnote{These are the abbreviations for  
	$(\exists^0 x^0_j\leq x^0_k)$ 
	and
	$(\forall^0 x^0_j\leq x^0_k)$, with $j$ and $k$ in $\bbN$.  The omission of 
	subscripts and/or superscripts somewhat increases 
	readability without creating  ambiguity.} 
There is no bounded quantification over the sort $\bbNN$. 

\newpar{} The formulas of second-order arithmetic are also called  \emph{analytical}\index{formula!analytical} formulas. 
The  \emph{arithmetical}\index{formula!arithmetical} formulas are the formulas of second-order arithmetic 
that do not involve any function quantifiers. 

\newpar{} \label{S.Sigma-1-n} For $n\geq 1$ the $\bSigma^1_n$ and $\bPi^1_n$ formulas  are defined recursively 
as follows. 
A formula is $\bSigma^1_n$ if it is of the form 
\[
(\exists^1 x_1)(\forall^1 x_2) (\exists^1 x_3)\dots (Q x_n) \varphi(\bar x, \bar y)
\]
where $Q$ is $\exists^1$ if $n$ is odd and $\forall^1$ if $n$ is even, 
and $\varphi$ is an arithmetical formula, possibly with parameters in $\bbNN$.  
A formula is $\bPi^1_n$ if it is of the form 
\[
(\forall^1 x_1)(\exists^1 x_2) (\forall^1 x_3)\dots (Q x_n) \varphi(\bar x, \bar y)
\]
where $Q$ is $\forall^1$ if $n$ is odd and $\exists^1$ if $n$ is even, 
and $\varphi$ is an arithmetical formula,  possibly with parameters in $\bbNN$.

The following lemma is proved by a coding argument showing that quantifiers over $\bbN$  can be absorbed into quantifiers over $\bbNN$. 

\newpar{} \label{L.Sigma1n-formula} \begin{Lemma}
	Suppose that $(\sfX_j,\Sigma_j)$, for $j<n+1$,  are standard Borel spaces and that $\sfY\subseteq \prod_{j<n+1} \sfX_j$ belongs to the $\sigma$-algebra generated by the rectangles with sides in $\Sigma_j$. 
	
	Then every formula of the form 
	\[
	(\exists^1 x_1)(\forall^1 x_2) (\exists^1 x_3)\dots (Q x_n) \varphi(\bar x, y)\in \sfY
	\]
	where $Q$ is $\exists^1$ if $n$ is odd and $\forall^1$ if $n$ is even, is $\bSigma^1_n$ and every formula of the form 
		\[
	(\forall^1 x_1)(\exists^1 x_2) (\forall^1 x_3)\dots (Q x_n) (\bar x, y)\in \sfY, 
	\]
	where $Q$ is $\forall^1$ if $n$ is odd and $\exists^1$ if $n$ is even, is $\bPi^1_n$. \qed 
	\end{Lemma}

\newpar{} Another way to state this is that $\bSigma^1_{n+1}$ formulas are logically equivalent to the statements of the form `$\sfA\neq \emptyset$' for a $\bPi^1_n$ set $\sfA$ while $\bPi^1_{n+1}$ formulas are logically equivalent to the statements of the form `$\sfA=\emptyset$' for a $\bSigma^1_n$ set~$\sfA$. 
If $\sfA$ is a Borel set, then asserting that $\sfA\neq \emptyset$ is $\bSigma^1_1$ and asserting that $\sfA=\emptyset$ is $\bPi^1_1$.

\subsection{\ZF{} and its models} The following discussion is well known to logicians and can be found in all sufficiently advanced books in logic in one form or another. 
Every mathematical statement $\varphi$ that can be formalized within \ZF{} (to the best of our knowledge, this includes all mathematical statements) either has a proof in \ZF{} or the completeness theorem of the first-order logic implies that there is a model of  \ZF+$\lnot\varphi$ . The latter cannot be taken as an indication that  \ZF+$\lnot\varphi$ is necessarily a reasonable theory. 
 
\newpar{} For example, if $\varphi$ is the G\"odel sentence that codes the assertion `ZF is consistent', then $\varphi$ has a proof in \ZF{} if and only if it is false. Thus, if \ZF{} has a model (i.e., if it is a  consistent theory), then it has a model $M$ in which $\lnot\varphi$ holds. Such $M$ contains a nonstandard natural number, a code for a `proof' of inconsistency of ZF. 
 In particular, $M$ is not well-founded (and even its set of natural numbers is not well-founded). Sentences consistent with \ZF{} that hold only in ill-founded models are generally considered to be false, albeit not provably false in ZF; see the last paragraph of \ref{Discussion.Absoluteness} for additional examples. 

 \newpar{}
Most natural models of \ZF{} are the ones that are transitive.  A model $M$ of \ZF{} is \emph{transitive} if the `$\in$' relation in $M$ is interpreted  in the natural way and every element of $M$ is a subset of $M$. 
While models obtained by the L\" owenheim--Skolem method are rarely transitive, every well-founded model of \ZF{} is isomorphic to one that is transitive via the \emph{Mostowski collapse} (see e.g., \cite[Theorem~A.6.2]{farah2019combinatorial}). 
Every transitive model is correct about the set of natural numbers. Since a proof of the inconsistency of \ZF{} (or of any other statement) can be coded by a natural number,  a transitive model cannot satisfy the sentence asserting that \ZF{} is inconsistent. 

\newpar{} If $\varphi(x,\bar a)$ is a first-order formula of the language of \ZF{} with parameters $\bar a$, $M$ is a transitive model of \ZF{} (or of a large enough fragment of ZF) such that $\bar a$ belongs to $M$, then for $b\in M$ we write 
\[
\varphi(b,\bar a)^M
\]
for the sentence obtained from $\varphi$ by replacing free occurrences of $x$ with~$b$ and restricting all of its quantifiers to $M$. This is the \emph{interpretation} of $\varphi(b,\bar a)$ in $M$. 

If $\sfA=\{x\mid \varphi(x,\bar a)\}$, then the interpretation of $\sfA$ in $M$ is defined to be  
\[
\sfA^M=\{b\in M\mid \varphi(b,\bar a)^M\}.
\]
\newpar {}\begin{Convention}
Among set theorists, the reals are commonly identified with the Cantor set (which is identified with $\cP(\bbN)$) or the Baire space~$\bbN^\bbN$.  By Kuratowski’s theorem (Theorem~\ref{T.Isomorphic}), every two uncountable standard Borel spaces are Borel-isomorphic.    
Suppose that $M$ is a transitive model of \ZF{} and $\varphi(x)$ is a formula such that $X=\{x\mid \varphi(x)\}$ is a subset of $\cP(\bbN)$ (i.e., a set of reals). What is the relationship between the sets  $X^M=\{x\in M \mid \varphi^M(x)\}$ and  $X=\{x\mid \varphi(x)\}\cap M$? If $\varphi$ has only bounded quantifiers, then a proof by induction on complexity of $\varphi$ shows that $X^M=X\cap M$, or in other words, that every $x\in M$ satisfies $\varphi^M(x)$ if and only if $\varphi(x)$ holds.  A moment of thought shows that, if $\varphi$ is sufficiently complex, there is no obvious reason why one should have $X\cap M=X^M$ in general. 
\end{Convention}

\newpar{}\begin{Definition}
	A subset $\sfA$ of a standard Borel space defined by a formula in parameters $\bar a$ is \emph{absolute for transitive models of ZF} if for every transitive model $M$ of \ZF{} such that $\bar a$ belongs to $M$ we have $\sfA^M=\sfA\cap M$. 
	
	A formula $\varphi$ (possibly with parameters $\bar a$) is \emph{absolute for transitive models of ZF} if for every transitive model $M$ of \ZF{} such that the parameters of $\varphi$ belong to~$M$,  $\varphi^M$ is equivalent to $\varphi$.  
\end{Definition}
	
	\newpar{} This definition naturally relativizes to any reasonable subclass of the class of transitive models.\footnote{Proceeding in the opposite direction, one could consider absoluteness for a larger class of models, as long as they satisfy  $\bbR^M=\bbR\cap M$. This follows from transitivity, but a weaker condition of being a a so-called $\omega$-model (i.e., satisfying $\bbN^M=\bbN$) suffices for this. This observation can be useful but we will not need it here.} In particular, we can talk about formulas and sets \emph{absolute for transitive models of \ZF{} that contain all countable ordinals}. 

\newpar{}\label{Ex.Delta12WO} The question of absoluteness is rather intricate. For example, in G\"odel's constructible universe $L$ (e.g., \cite{Ku:Set}) there is a $\Sigma^1_2$ well-ordering of the reals.\footnote{This is a lightface $\Sigma^1_2$, meaning that the well-ordering is defined without using any parameters.}  More precisely, there is a Borel (even $G_\delta$) $\sfB\subseteq \bbR^4$ such that 
\[
x\leq^L y \text{ if and only if } (\exists z)(\forall w) (x,y,z,w)\in \sfB
\]
defines a well-ordering of the reals. (In other words, if we take the projection of $\sfB$ to $\bbR^3$, then the projection of its complement to $\bbR^2$ is the graph of a well-ordering of $\bbR$.)

The statement that $\leq^L$ is a linear ordering is obviously $\Pi^1_3$,\footnote{Since no parameters are needed for this definition, we are using the lightface notation.}\,\footnote{The statement is $(\forall x\in \bbR)(\forall y\in \bbR) (x\leq^L y\lor y\leq^L x)$.} and a little reflection using the fact that $\bbR^\bbN$ has a standard Borel space structure shows that even the statement that $\leq^L$ is a well-ordering of $\bbR$ all of whose proper initial segments are countable is $\Pi^1_3$.\footnote{$(\forall (x_n)\in \bbR^\bbN)(\exists n) x_{n+1}\not\leq^L x_n$.} Each one of these $\Pi^1_3$ statements is true in $L$, but they both fail e.g., in Cohen's original model for the negation of the Continuum Hypothesis (cf. \ref{Ex.Luzin}). In other words, we have an example of a $\Pi^1_3$ statement that may not be absolute between transitive models of \ZFC{} that contain all countable ordinals.\footnote{If $V=L$, then $L$ is the only transitive model of \ZFC{} that contains all ordinals. hence every statement is absolute between such models.}

\subsection{Trees and absoluteness} 
\newpar{} Lemma~\ref{L.TreeRank} below is the basis of all known absoluteness results. 

\newpar{}\begin{Definition}
If $\sfX$ is a set, then a tree on $\sfX$ is a set $\sfT\subseteq \bigcup_n \sfX^n$ which is hereditary, in the sense that $t\in \sfT$ implies  all of its initial segments belong to $\sfT$. 
 For $s$ and $t$ we write $s\sqsubset t$ if $s$ is a proper initial segment of $t$. 
A \emph{branch} of $\sfT$ is $x\in \sfX^\bbN$ such that $x\rs m\in \sfT$ for all $m\in \bbN$. 
It is convenient to turn trees upside down, and a tree is called \emph{well-founded} if it has no infinite branches. 
\end{Definition}

\newpar{}\begin{Lemma} \label{L.TreeRank} If $M$ is a transitive model of \ZF{} and $\sfT\in M$ is a tree on some  set $\sfX$, then the assertion `$\sfT$ is well-founded' is absolute between $M$ and any transitive model $N$ of \ZF{} (including the universe itself) that includes~$M$. 
	\end{Lemma}
	
	\begin{proof} Since $M$ is transitive, we have $\sfX\subseteq M$. If $\sfT$ has an infinite branch $x$ in $M$, then $x$ is still an infinite branch in $N$. Assume $\sfT$ is well-founded in~$M$.  Recursively define
	rank function $\rho_\sfT$ from $\sfT$ into the ordinals by 
	\[
	\rho_\sfT(s)=\sup\{\rho_\sfT(t): s\sqsubset t, t\in \sfT\}+1. 
	\]
	The supremum of the range of $\rho_\sfT$ is an ordinal of cardinality no greater than that of~$\sfT$, and since $M$ is a model of \ZF{} this ordinal is included in $M$. 
	
	Since $\sfT$ is well-founded, the domain of $\rho_\sfT$ is equal to $\sfT$ and $s\sqsubset t$ implies $\rho_\sfT(s)>\rho_\sfT(t)$.  If $x$ is an infinite branch of $\sfT$ in $N$, then $\rho_\sfT(x\rs n)$, for $n\in \bbN$, is an infinite decreasing sequence of ordinals; contradiction. 
	\end{proof}
	
\newpar{} \begin{Theorem}\label{T.Analytic}
	Every $\bSigma^1_1$ set is absolute for transitive models of \ZF{} that contain the parameter from which the set is defined. 
	Every $\bSigma^1_1$ statement is absolute for transitive models of \ZF{} that contain the parameter from which the set is defined.
	\end{Theorem}

\begin{Proof} It suffices to prove the statement for $\bSigma^1_1$ sets. By transitivity (of logical equivalence, pun not intended), it also suffices to prove absoluteness between a fixed transitive model of $\ZF$ and the universe.  We first prove  absoluteness for $\bSigma^1_1$ subsets of the Baire space $\bbNN$. For every analytic subset $\sfA$ of $\bbNN$ there is a closed $\sfB\subseteq (\bbNN)^2$ such that $\sfA$ is the projection of $\sfB$. Let 
	\[
	\sfT=\{(x\rs n,y\rs n): (x,y)\in \sfB, n\in \bbN\}. 
	\]
	Then $\sfT$ is a tree on $\bbN\times \bbN$ and 
	\begin{equation}\label{eq.A=p[T]}
	\sfA=\{x\in \bbNN: (\exists y\in \bbNN) (x\rs n,y\rs n)\in \sfT\}. 
	\end{equation}
In other words, writing (by $|s|$ we denote the length of $s$)
\[
\sfT_x=\{s: (x\rs |s|,s)\in \sfT\}
\]
we have that $\sfA$ is the set of all $x$ such that $\sfT_x$ is ill-founded. By Lemma~\ref{L.TreeRank}, the latter assertion is absolute for transitive models of \ZF{} and therefore for $x\in M$ we have that $x\in \sfA^M$ if and only if $x\in \sfA$, as required. 

For subsets of an arbitrary Polish space $\sfX$,  Theorem~\ref{T.Isomorphic} implies that~$\sfX$ is Borel-isomorphic to $\bbNN$ and that  the isomorphism is definable from an enumeration of a countable generating subset of $\Sigma$, and absolute for transitive models of ZF. Since Borel-isomorphism sends $\bSigma^1_n$ sets to $\bSigma^1_n$ sets, the conclusion follows. 
\end{Proof}

The following is Shoenfield's Absoluteness Theorem (for a reason why this theorem matters see Theorem~\ref{T.Provability}). 

\newpar{}\begin{Theorem}\label{T.Shoenfield}
	Every $\bSigma^1_2$ set, and every $\bSigma^1_2$ statement,  is absolute for transitive models of \ZF{} that contain the parameter from which the set is defined and all countable ordinals. 
\end{Theorem}

\begin{ProofOf}{Guide to the Proof of Theorem \ref{T.Shoenfield}} As in the proof of Theorem~\ref{T.Analytic}, it suffices to prove the statement for $\bSigma^1_2$ subsets of $\bbNN$.  The proof in this case  can be found in the literature (\cite[Corollary~7.15]{drake1974set}, \cite[Theorem~8F.10]{Mo:Descriptive}, \cite[Theorem~13.15]{kanamori2008higher};  in \cite[\S 36.D]{Ke:Classical} the key object for the proof, a Shoenfield tree, is defined for other purposes).  
	The salient point is that the Shoenfield tree is a tree $\sfT$ on $\bbN\times \aleph_1$ such that $\sfA$ satisfies the analog of equation~\eqref{eq.A=p[T]}, $x\in \sfA$ if and only if $\sfT_x$ has an infinite branch, and Lemma~\ref{L.TreeRank} applies . 
\end{ProofOf}

\newpar{}Absoluteness of $\bPi^1_2$ sets and statements follows by taking complements and negations. 
It should be noted that $\bSigma^1_1$ sets (and statements) are absolute for transitive models of \ZF{} that contain the parameter from which the set is defined, even if the models are countable.
Example \ref{Ex.Delta12WO} shows that absoluteness for $\bSigma^1_3$ sets can fail even between models of \ZFC{} that contain all ordinals. Absoluteness for all $\bSigma^1_3$ sets, and even all projective sets, can also be achieved, but for a more restrictive class of models and under the assumption of the existence of large cardinals (see the remark following \cite[Theorem~15.6]{kanamori2008higher}).

\newpar{} The requirement in Theorem~\ref{T.Shoenfield} that the model  be well-founded and contain all countable ordinals is not as restrictive as it may appear to be. The standard set-theoretic methods for proving independence with $\ZF$---that is, forcing and passing to inner models (see \cite{Ku:Set}, \cite{Jech:SetTheory}) start from a model $M$ of $\ZF$ ($\ZFC$,  a sufficiently large fragment of one of these theories\dots)  and produce a larger model $N$ of $\ZF$ (and the statement whose relative consistency with $\ZF$ is being proved) that has the same ordinals as~$M$. Therefore, in $N$ the model $M$ contains \emph{all} ordinals and therefore satisfies the assumptions of Shoenfield's theorem as applied in $N$. Thus $M$ and $N$ satisfy the same $\bSigma^1_2$ statements. 

\newpar{} Even more is true. 
Perhaps surprisingly, external ill-foundedness does not prevent a model from carrying out correct internal reasoning about well-foundedness. If $M$ is a model of \ZF{} then $M$ satisfies the statement `there is no decreasing sequence of ordinals' even if it is externally ill-founded  (see \S\ref{S.Ultraproduct} below for an example). 
It also satisfies the statement `there are uncountably many ordinals', because this is a theorem of \ZF, 
and therefore has its own $\aleph_1$, denoted $\aleph_1^M$. Note that it is a countable  ordinal if $M$ is countable and transitive, but  
externally, $\aleph_1^M$ need not even be an ordinal when $M$ is ill-founded.
Let $N$ be any submodel of $M$ that is a model of \ZF, 
internally transitive (that is, $M\models (\forall x\in N)(x\subseteq N)$)  and contains all of~$M$'s ordinals smaller than $\aleph_1^M$. 
Since Shoenfield's Absoluteness Theorem is provable in \ZF, every model of \ZF{} satisfies its conclusion about internally transitive submodels. Hence $M$ satisfies that whenever $N\subseteq M$ is internally transitive  and contains all ordinals below $\aleph_1^M$, the models $M$ and $N$ agree on all $\bSigma^1_2$ statements whose parameters belong to $N$.  
This discussion proves the following.

\newpar{}\begin{Corollary}\label{C.IllFounded}
	Every $\bSigma^1_2$ or $\bPi^1_2$ statement $\varphi$ is absolute between any two models $M_1$ and $M_2$ of \ZF{} which are transitive and contain all countable ordinals (and parameters of $\varphi$) in the interpretation of "transitive" and "all countable ordinals" within some model $M$ of \ZF{} that has both $M_1$ and $M_2$ as elements, or as definable classes.  \qed
\end{Corollary}

\newpar{}\label{S.L[x]} We need one additional ingredient in order to handle parameters.  Suppose $x$ is a code for the parameters  occurring in~a  formula~$\varphi$ of second-order arithmetic (\ref{S.Sigma-1-n}). This code can be chosen in any recursively defined Polish space, and  it will be convenient to choose $x$ in $\cP(\bbN)$.  For a real $x$, $L[x]^M$ consists of all sets constructible from $x$ as computed in $M$ (\cite[Definition 11.6.29]{Ku:Set}). Since $x$ is a set of ordinals, by \cite[Exercise 11.6.30]{Ku:Set}, $L[x]^M$ is a model of \ZFC. We therefore have the following consequence of Corollary~\ref{C.IllFounded}. 

\newpar{}\begin{Corollary}\label{C.IllFounded-with-parameters}
Every $\bSigma^1_2$ or $\bPi^1_2$ statement  $\varphi$  with parameters $x$ is absolute between  $M$ and $L[x]^M$ for any model $M$ of \ZF.
\end{Corollary}

\newpar{} The reader may wonder about the role of the assumption in Theorem~\ref{T.Shoenfield} that models contain all countable ordinals. Consider the statement $\varphi$:  `There exists a well-founded, countable model of \ZFC'. This is a $\bSigma^1_2$ statement. If it is true, then there is a smallest model of \ZFC, $M$, in which it is true, as otherwise one can find a strictly decreasing sequence of well-founded models. Then $\varphi$ is not true in $M$, although it holds in $V$. 
\goodbreak
\subsection{Coding} 
\newpar{} If $M\subseteq N$ are transitive models of ZF, then a Borel set in $M$ is not necessarily a Borel set in $N$. This is for example the case if $M$ is the G\"odel's universe $L$ and $N$ is Cohen's model in which the Continuum Hypothesis fails. Then $\bbR^M$, the set of real numbers in $M$, is an uncountable set of reals in $N$ whose cardinality is $\aleph_1<2^{\aleph_0}$ and therefore it (or any of its uncountable Borel subsets) cannot be Borel in $N$.

Coding of Borel subsets of $\bbR$ or of any other Polish space was developed in \cite{So:Model}, but we will not need its full power. The Baire space $\bbNN$ has a natural code, the tree $\bbNlN$ of all finite sequences of natural numbers. Thus the interpretation of $\bbNN$ in a transitive model $M$ of \ZF{} is $(\bbNN)^M$, the set of all branches of the tree $\bbNlN$ that belong to $M$.

\subsection{Provability}

The following consequence of Shoenfield's absoluteness is folklore among logicians, but we could not find a convenient reference. Since it is particularly useful outside logic, we record it here (for the definition of $L[x]$ see \ref{S.L[x]}).

\newpar{}\begin{Theorem}\label{T.Provability}
Every $\bSigma^1_2$ statement, and every   $\bPi^1_2$ statement,  that is provable in \ZFC{} or even in \ZF\, +\, $V=L$  (or $V=L[x]$, if $\varphi$ has a parameter  coded by $x$) is provable in \ZF.
\end{Theorem}

\begin{Proof} Fix $\varphi$ that is $\bSigma^1_2$ or $\bPi^1_2$ and provable in \ZF\,+\,$V=L[x]$, where $x\in \cP(\bbN)$ codes the parameters of $\varphi$. 
		Fix a model $M$ of \ZF{} that contains $x$. This code can be chosen in any recursively defined Polish space, and  it will be convenient to choose $x$ in $\cP(\bbN)$. By our assumption, $\varphi$ holds is $L[x]^M$ (i.e., $L[x]$ as computed in $M$) and by  Corollary~\ref{C.IllFounded-with-parameters}, $\varphi$ is absolute between  $M$ and $L[X]^M$. 	
	Since $M$ was an arbitrary model that contains $x$, this proves that $\varphi$ holds in every model of \ZF{} that contains $x$. 
By the Completeness Theorem of first-order logic, it follows that $\varphi$ is a theorem of \ZF. 
\end{Proof}

\newpar{}Theorem~\ref{T.Provability}  says, among other things,  that Choice cannot be essential in any proof of a $\bSigma^1_2$ or $\bPi^1_2$ statement. If such statement can be proved using the Axiom of Choice, then there exists a proof that avoids Choice altogether. 

\newpar{} 
A standard  modification of the proof of Theorem~\ref{T.Provability} shows the following (the assumption \eqref{T.Provable.3} is weaker than either of the previous two since both \OCAT+\MA{} and $\diamondsuit_{\aleph_1}$ hold in forcing extensions of any model of ZFC, see \cite[Theorem V.4.1]{Ku:Set} for \MA{} and \cite[p. 142]{velickovic1992applications} for \OCAT+\MA). 

\newpar{}\begin{Theorem}
	Assume $\varphi$ is a  $\bSigma^1_2$ statement or a $\bPi^1_2$ statement  and that in addition it satisfies one of the following. 
	\begin{enumerate}
		\item $\varphi$ follows from \ZFC+\OCAT+\MA. 
		\item $\varphi$ follows from \ZFC+$\diamondsuit_{\aleph_1}$. 
		\item \label{T.Provable.3} $\varphi$ holds in some forcing extension of the constructible universe $L$.
	\end{enumerate} 
	Then $\varphi$ is provable in \ZF. 
	\qed 
\end{Theorem}

\newpar{} Being true in all transitive models  of  \ZFC{} and therefore absolute between all transitive models of \ZFC{} arguably suggests the truth of the given statement. This property however does not guarantee being provable in \ZFC{}. For example, take the sentence that says `\ZFC{} is consistent'. Via G\"odel coding, this is  a $\Pi^0_1$ statement (it says `for every natural number $n$, $n$ does not code a proof that \ZFC{} is inconsistent), in particular  it is arithmetical. Every arithmetical
statement is absolute between transitive models of \ZFC{}. Actually, arithmetical statements are absolute between models that do not contain any nonstandard natural numbers. Equivalently, these are the models whose $\omega$ is well-founded, and such models are called $\omega$-models. However, G\"odel's Incompleteness Theorem implies that this sentence is not provable in \ZFC{} unless \ZFC{} is inconsistent (we take the latter statement on faith).
The fact is that, by completeness of predicate logic, there exists a model of \ZFC{} in which the sentence that asserts `\ZFC{} is inconsistent' holds. In this model there is a nonstandard integer that codes a `proof' of $0=1$ from \ZFC{}.
Such model cannot be an $\omega$-model.

\newpar{}\label{S.Ultraproduct}
Producing a model of an arbitrarily large finite fragment \ZFCstar{} of \ZFC{}\footnote{By G\"odel's Theorem one cannot produce a model of \ZFC{} within \ZFC{}, since this would show that \ZFC{} implies its own consistency.} that is not an $\omega$-model is easy. 
Let $M$, be a transitive models of \ZFCstar{}.  If $\cU$ is a nonprincipal ultrafilter on $\bbN$, then the ultrapower $\prod_{\cU} M$ is (by \L o\'s's Theorem) a model of \ZFCstar{}, but it contains nonstandard integers. Note that \L o\'s's Theorem implies that in this way one cannot produce  a model of the theory \ZFCstar+`\ZFC{} is inconsistent'. 

\newpar{} Another application of absoluteness is worth mentioning. If $\varphi$ is a projective statement that follows from \ZFC{} together with the Continuum Hypothesis, \CH, then $\varphi$ follows from \ZFC. This is because there exists a forcing notion (the Levy collapse of $2^{\aleph_0}$ to $\aleph_1$ by countable conditions) that forces 
\CH{} without adding any new real numbers, and therefore without changing the truth of projective statements. This well-known trick has been applied to operator algebras in  \cite[Proposition 5.19]{eagle2015saturation}, \cite[Lemma~2.2]{gould2026uncountable}, and 
\cite[Lemma~5.7]{gould2026saturated}.

	\section{Applications}
	\subsection{Separable \cstar-algebras and absoluteness}

\newpar{} \label{P.L.states} 
We provide an alternative proof of Lemma~\ref{L.states}.

\begin{Proof} Fix a unital separable \cstar-algebra $A$ and $x\in A$. 
	We need to prove that there is a state $\varphi$ on $A$ such that $\varphi(x^*x)=\|x\|^2$. This is a result of \ZFC, and in order to conclude that it can be proved in ZF, by Theorem~\ref{T.Shoenfield} it will suffice to prove that this is a $\bPi^1_2$ statement with code for $A$ as a parameter.

	Since the map $(a,\varphi)\mapsto \varphi(a)$ is jointly continuous on $A\times \sfS(A)$, the set 
	\[
	Z=\{(a,\varphi)\in A\times \sfS(A)\mid \varphi(a^*a)=\|a\|^2\}
	\]
	is a closed subset of $A\times \sfS(A)$, which is a product of two Polish spaces ($A$ is Polish by the assumption and $\sfS(A)$ is Polish by Lemma~\ref{L.S(A)}).
	
	Therefore the statement $(\forall a\in A)(\exists \varphi\in \sfS(A))\varphi(a^*a)=\|a\|^2$ is $\bPi^1_2$. Thus this is a  theorem of \ZF{} by Theorem~\ref{T.Shoenfield}. 
\end{Proof}

\newpar{}\label{Pf.Spectral}  We can now give a proof of the Spectral Mapping Theorem for polynomials (Theorem~\ref{T.SpectralMapping}).

\begin{Proof}
	Fix $n\geq 1$ and  commuting normal elements $a_1,\dots, a_n$  of a \cstar-algebra $A$.  We first need to prove that $\jsigma(\bar a)$ is nonempty.
	
Since the standard proof that the joint spectrum of  an $n$-tuple of normal operators does not change when passing to a unital \cstar-subalgebra does not require any Choice,   
	by replacing $A$ with $\cst(\bar a,1)$ we may assume that it is separable. Fix a countable dense subset $D$ of $A$. Let 
	\[
\textstyle	B=\{(\bar \lambda, \bar b)\in \bbC^n\times D^n\mid \|\sum_{j\leq n} b_j (a_j-\lambda_j 1)-1\|<1 \}. 
	\]
We claim that $\bar \lambda\in \jsigma(\bar a)$ if and only if $(\forall \bar b\in D^n)(\bar \lambda, \bar b)\notin B$. Suppose that $\bar \lambda\in \jsigma(\bar a)$ and fix $b\in D^n$. Since $a_j-\lambda_j 1$, for $j\leq n$,  generate  a proper ideal, $x=\sum_{j\leq n} b_j (a_j-\lambda_j 1)$ is not invertible, and in particular $\|x-1\|\geq 1$. Therefore $(\bar \lambda, \bar b)\notin B$. 
Conversely, suppose that $\bar \lambda\notin \jsigma(\bar a)$. Then the ideal generated by $a_j-\lambda_j 1$, for $j\leq n$, is improper and $\sum_j d_j (a_j-\lambda_j 1)=1$. Since~$D$ is dense, we can find $\bar b$ such that $(\bar \lambda ,\bar b)\in B$.

 Since $\bar b$ ranges over the countable set $D^n$, and since the section of $B$ at each $\bar b$ is an open subset of $\bbC^n$, we conclude that $\bbC^n\setminus \jsigma(\bar a)$ is an open subset of $\bbC^n$. Therefore $\jsigma(\bar a)$ is closed, and the assertion that it is nonempty is~$\bSigma^1_1$, and therefore absolute by Theorem~\ref{T.Shoenfield}. Since  it is a theorem of \ZFC{} (\cite{Arv:Short}), the conclusion follows. 
Since the joint spectrum of $\bar a$ is included in the bounded set $\{\bar \lambda : |\lambda_j|\leq \|a_j\|$ for all $j\leq n\}$, it is a closed and bounded  subset of $\bbC^n$ and therefore compact. 

It remains to prove that if $f$ is a complex *-polynomial in $2n$ variables, then 
\begin{multline*}
\sigma(f(a_1,\dots, a_n, a_1^*,\dots, a_n^*))\\
=\{f(\lambda_1, \dots, \lambda_n, \bar\lambda_1, \dots, \bar\lambda_n)\mid (\lambda_1,\dots, \lambda_n)\in \jsigma(\bar a)\}. 
\end{multline*}
Fix such $f$. Then the sets 
	$X=\sigma(f(a_1,\dots, a_n, a_1^*,\dots, a_n^*))$ and $\jsigma(\bar a)$ are compact by the first part of this Theorem. The continuous image of the latter, $Y=\{f(\lambda_1, \dots, \lambda_n, \bar\lambda_1, \dots, \bar\lambda_n)\mid (\lambda_1,\dots, \lambda_n)\in \jsigma(\bar a)\}$ is therefore also compact. 

	Finally, the assertion that $X=Y$ is $\bPi^1_1$ because it asserts that the symmetric difference of $X$ and $Y$ (a Borel set) is empty, i.e., that $(\forall \lambda\in \bbC)(\lambda\notin X\Delta Y)$. Therefore, Theorem~\ref{T.Shoenfield} implies that these two sets are equal. 
\end{Proof}

\newpar{} No Choice is needed to prove the Stone--Weierstrass Theorem (\ref{S.CFC}).
Combining this with Theorem~\ref{T.SpectralMapping}, we obtain continuous functional calculus for $n$-tuples of commuting normal operators. 

\newpar{}
\begin{Proposition}\label{FullFunctCalc}
	Let $A$ be a unital \cstar-algebra, $n\geq 1$, and let $\bar a$ be an $n$-tuple of commuting normal operators.  Then $p\mapsto p(\bar a)$ ($p$ is a complex *-polynomial in $n$ commuting variables)
	extends to an isometric
	isomorphism from $C(\jsigma(\bar a))$ onto the uniformly closed subalgebra of $A$
	generated by $\bar a$ and~1.  If $\bar a$ is closed under adjoints, this is $\cst(\bar a,1)$. 
\end{Proposition}

\newpar{}\label{S.HBseparation}
We can now give a proof of the Hahn--Banach separation theorem promised in \ref{P.HBetc} \eqref{5.BS.HBs}. 

\begin{Proof} Fix disjoint convex subsets $A$ and $B$ of a separable normed space~$X$ such that $A$ is open. By replacing $B$ with its closure, we may assume   $B$ is closed. As open and closed subsets of a Polish space $X$, both $A$ and $B$ are coded by real parameters.

In the following we write $X_m$, $X^*_m$ for the $m$-balls in $X$ and $X^*$. 
While $X^*_m$ is not necessarily norm-separable, it is compact and metrizable in the weak$^*$-topology, and we will consider it with the corresponding Borel structure. 

The duality map  from $X\times X^*$ into the scalars, $(\xi,\varphi)\mapsto \varphi(\xi)$,  is jointly continuous on  $X_m\times X^*_n$, when $X$ is considered with the norm topology and~$X^*$ with the weak*-topology,  for all $m,n\geq 1$. This is because for all $\xi,\eta$ in $X_m$ and all $\varphi,\psi$ in $X^*_n$ we have that 
\[
|\psi(\eta)-\varphi(\xi)|
\leq |\psi(\eta-\xi)|+|(\psi-\varphi)(\xi)|
\leq n\|\eta-\xi\|+|(\psi-\varphi)(\xi)|.
\]
The first term is uniformly small when $\eta$ is norm-close to $\xi$, and the second is controlled by a basic weak-* neighbourhood evaluated at $\xi$.
The set 
\begin{align*}
Z=\{(\varphi, \xi, \eta,r): \varphi\in X^*, \ &\xi\notin A\text{ or }\xi\in A\text{  and }\Re(\varphi(\xi))<r, \text{ and }   \\
&\eta\notin B \text{ or  }\eta\in B\text{ and }r\leq \Re(\varphi(\eta)))\}
\end{align*}
is therefore a Borel subset of $X^*\times X^2\times \bbR$, and the existence of a linear functional that separates $A$ from $B$ is equivalent to 
\[
(\exists\varphi)(\exists r)(\forall \xi,\eta\in X)(\varphi,\xi,\eta,r)\in Z. 
\]
This is a $\bSigma^1_2$ statement (with codes for $X$, $A$, and $B$ as parameters), hence it is absolute between transitive models of \ZF{} that contain all countable ordinals. Since it is a theorem of \ZFC, it is a theorem of \ZF.  
\end{Proof}

\newpar{}\label{S.KreinMilman}
We conclude  this section with a proof of the Krein--Milman theorem promised in \ref{P.HBetc}. 

\begin{Proof}
Suppose that $K$ is a compact convex subset of the dual of a separable Banach space $X$. 
By the Banach--Alaoglu theorem, $K$ is a compact metrizable space.  Fix a compatible metric $d$ on $K$. 
The extreme boundary of $K$ is equal to 
\[
\partial K=\bigcap_{n\in \bbN}\{x\in K\mid (\forall y,z\in K)x=(y+z)/2\text{ implies } d(y,z)<1/n\}.
\]
Since the set $\{(y+z)/2: y,z\in K, d(y,z)\geq 1/n\}$ is compact (as a continuous image of a compact set), the set $\partial K$ is a $G_\delta$ set. The assertion that this~$G_\delta$ set is nonempty is $\bSigma^1_1$ (it involves a code for $K$ as a parameter, thus the boldface), and therefore absolute between transitive models of \ZF{} by the $\bSigma^1_1$ Absoluteness Theorem (Theorem~\ref{T.Analytic}). Since it is a theorem of \ZFC{} that every compact convex set has an extremal point, this is also a theorem of \ZF{} by Theorem~\ref{T.Provability}.  

In order to prove that every  separable compact convex set in a locally convex space is the closed convex hull of its extreme points, assume otherwise. Consider the closed convex hull  $C$ of $\partial K$ and fix $x\in K\setminus C$. By the Hahn--Banach separation theorem (Proposition~\ref{P.HBetc} \eqref{5.BS.HBs}) there are  a continuous affine function $\phi\colon K\to \bbR$ and $r\in \bbR$ such that $\phi(x)\leq r<\phi(y)$ for all $y\in C$. Then $K$ has a face exposed by $\phi$ that is disjoint from $C$. By the first part of the proof, this face has an extremal point, and such point is an extremal point of $K$;  contradiction. 
\end{Proof}

\subsection{Other applications}\label{S.other}

We record some other applications of Shoenfield's Absoluteness theorem.

\newpar{} \label{P.CodingC*} We will first need some standard facts about coding separable \cstar-algebras. A discussion analogous to the following applies to Banach spaces, Banach algebras, and separable metric structures in general.  
A Borel space of concrete separable \cstar-algebras was defined in \cite{Kec:C*}. The space $\hat \Xi$  of abstract separable \cstar-algebras defined in \cite[\S 2.4]{FaToTo:Turbulence} will be more convenient for our purposes.  
Its elements are countable $(\bbQ+i\bbQ)$-algebras equipped with a norm and involution. 
As noted in \cite[p. 8]{FaToTo:Turbulence}, axioms for being a normed *-algebra and the \cstar-equality correspond to Borel conditions. By induction on quantifier rank one proves an analogous statement for every condition in continuous logic
(\cite{BYBHU}, as adapted to \cstar-algebras and other not necessarily bounded metric structures in \cite{Muenster}, see also \cite{hart2023an}). 
Therefore, for every sentence $\varphi$ in the language of \cstar-algebras, the set of  codes in~$\hat \Xi$ that correspond to separable algebras $A$ such that $\varphi^A=0$ is Borel.

This implies the following.

\newpar{}\begin{Lemma}\label{L.Axiomatizable.0}
	If $\cP$ is a property of separable \cstar-algebras, and it is provable \emph{in ZF} that $\cP$ is axiomatizable,  then the set of codes in $\hat \Xi$ that correspond to algebras that satisfy  $\cP$ is Borel. \qed 
\end{Lemma}

The assumption that the axiomatizability is provable in \ZF{} (and not merely in ZFC) is not vacuous; see Proposition~\ref{P.not-axiomatizable} below. 

\newpar{}
Recall that a \cstar-algebra $A$ is nuclear if and only if for every \cstar-algebra~$B$ the algebraic tensor product $A\odot B$ has a unique \cstar-norm and that it satisfies CPAP (completely positive approximation property) if the identity map on $A$ is nuclear, meaning that for every $\varepsilon>0$ and every finite $F\subseteq A$ there is $m$ and completely positive maps $\varphi:A\to M_m(\bbC)$ and $\psi: M_m(\bbC )\to A$ such that $\max_{a\in F} \|\psi\circ \varphi(a)-a\|<\varepsilon$. 
Nuclearity and CPAP are equivalent in \ZFC{} (see e.g., \cite[Lemma~IV.3.1.6]{BlackadarOperator}), but we don't know whether they are equivalent in ZF.

\newpar{}\begin{Lemma}\label{L.Axiomatizable} 
			\begin{enumerate}
				\item For each of the following properties of \cstar-algebras, the set of codes of separable \cstar-algebras that satisfy it is Borel: 
				Being finite, stably finite, purely infinite and simple. 
				\item \label{2.Axiom} The set of codes for separable \cstar-algebras that have a tracial state is $\bSigma^1_1$. 
				
				\item \label{3.Axiom} The set of codes for separable 
				nuclear \cstar-algebras is $\bPi^1_1$. 
				
				\item \label{4.Axiom} The set of codes for separable \cstar-algebras with CPAP is Borel. 
			\end{enumerate}

\end{Lemma}

\begin{Proof} For finite, stably finite, and purely infinite and simple one only needs to observe that the defining property of the class is already first order; see the proof  of  \cite[Theorem~2.5.1]{Muenster}. 

\eqref{2.Axiom}
Let $A$ be a code for a separable \cstar-algebra. 
Since the space $\sfS(A)$ of states of $A$ is a compact metrizable space (Lemma~\ref{L.S(A)}) and the evaluation function is continuous, the space of all tracial states of $A$ is closed and the assertion that it is nonempty is $\bSigma^1_1$. 

\eqref{4.Axiom}
The class \cstar-algebras with CPAP is Borel by a result of Effros (\cite[Theorem~1.2]{Kec:C*}).  This result is stated using the equivalence of nuclearity and CPAP, but from the proof it is clear that CPAP is used. An alternative proof of this, and Borelness of related approximation properties (such as for example small nuclear dimension) can also be extracted from \cite[Theorem~5.7.3]{Muenster}.  
		
		\eqref{3.Axiom}		
For nuclearity, 	fix a (code for) a separable \cstar-algebra $A$. 
Then $A$ is not nuclear if and only if there is a \cstar-algebra $B$ with $\|\cdot\|_{\max}\neq\|\cdot\|_{\min}$ on $A\odot B$.  For such a $B$, there is a $z=\sum_{k=1}^n x_k\otimes y_k\in A\odot B$ with $\|z\|_{\max}>\|z\|_{\min}$.  Then $\|\cdot\|_{\max}\neq\|\cdot\|_{\min}$ on $A\odot C^*(y_1,\dots,y_n)$.
Therefore, we may assume $B$ is separable.  Since both $A$ and $B$ are separable, $\|\cdot\|_{\min}$ is the norm in a GNS-representation by states of the form $\psi\otimes \eta$, where $\psi$ and $\eta$ are faithful states of $A$ and $B$.  

\newpar{}\begin{Claim} 
The set of all triples $(B,x,\varphi,\psi,\eta)$ where $B$ is a code for a separable \cstar-algebra, $x$ is a sum of $n$ elementary tensors in $A\odot B$ for some $n\in \bbN$, $\varphi$ is a state on $A\odot B$, $\psi$ and $\eta$ are faithful states on $A$ and $B$ and  
\[
\|\pi_\varphi(x)\|>\|x\|_{\psi\otimes \eta} 
\]
is Borel. 
\end{Claim}
\begin{Proof}
	First, states, faithful states, and finite sums of elementary tensors all have standard Borel codings. Second, each of the two GNS norms is a Borel function of the codes for states. 
\end{Proof}

Finally, since $\pi_\psi$ and $\pi_\eta$ are faithful, the right-hand side in the displayed formula of the Claim is equal to $\|x\|_{\min}$. 
	Therefore, the set of codes for non-nuclear \cstar-algebras is, as the projection of this set,~$\bSigma^1_1$. 
	This implies that the set of codes for nuclear \cstar-algebras is $\bPi^1_1$ as claimed. 
\end{Proof}

\newpar{}
Theorem~\ref{T.absoluteness.C*.separable} below implies that the trace problem and the UCT problem  are absolute between transitive models of ZFC. It however does not imply that one or both of these statements cannot be independent from ZF, or even ZFC. See the discussion in \ref{Discussion.Absoluteness}. 
As pointed out in the paragraph preceding Lemma~\ref{L.Axiomatizable}, we do not know whether in \ZF{} one can prove that CPAP is equivalent to nuclearity.  

\newpar{}\begin{Theorem}\label{T.absoluteness.C*.separable} Each one of the following statements is absolute between transitive models of \ZF{} that contain all countable ordinals. 
\begin{enumerate}
	\item \label{1.Abs} Every separable, stably finite, \cstar-algebra has a tracial state. 
	\item \label{nuclear-UCT} Every separable, nuclear \cstar-algebra satisfies the Universal Coefficient Theorem (UCT). 
	\item \label{CPAP-UCT} Every separable \cstar-algebra with CPAP satisfies the Universal Coefficient Theorem (UCT). 
\end{enumerate}
\end{Theorem}

\begin{Proof} \eqref{1.Abs}
By Lemma~\ref{L.Axiomatizable}, the set of codes in $\hat\Xi$  corresponding to separable \cstar-algebras that are stably finite and have no tracial state is $\bPi^1_1$. Thus the assertion that this set is empty is $\bPi^1_2$, hence absolute between transitive models of \ZF{} that contain all countable ordinals by Theorem~\ref{T.Shoenfield}. 

\eqref{nuclear-UCT} and \eqref{CPAP-UCT}
The set of codes for separable \cstar-algebras with UCT is $\bSigma^1_1$ by \cite[Theorem~3.27]{pi2026game}. By Lemma~\ref{L.Axiomatizable}, \eqref{CPAP-UCT}  asks whether a concrete Borel set is included in a concrete $\bSigma^1_1$ set and \eqref{nuclear-UCT} asks whether a concrete~$\bPi^1_1$ set is included in a concrete $\bSigma^1_1$ set. Thus each of these two statements is equivalent to asserting that every code for a separable \cstar-algebra belongs to a certain~$\bSigma^1_1$ set. Such statement is equivalent to a $\bPi^1_2$ statement, and therefore absolute for transitive models of \ZF{} that contain all countable ordinals by Shoenfield's Absoluteness Theorem, Theorem~\ref{T.Shoenfield}.  

\end{Proof}

\subsection{Reflection and $\sigma$-completeness of $\Sep(A)$}

\newpar{} In this section, $\sigma$-completeness is a standard set-theoretic notion, different from $\sigma$-completeness of metric spaces discussed elsewhere in this paper.

\newpar{} 
In the presence of the Axiom of Choice, many properties of nonseparable \cstar-algebras reflect to separable \cstar-subalgebras (see \cite[Property (SI)]{BlackadarOperator}, and also \cite[\S 7.3]{farah2019combinatorial}). These facts hinge on the directed system $\Sep(A)$ of separable \cstar-subalgebras of $A$  being \emph{$\sigma$-complete}, meaning that every countable increasing sequence of its elements has a supremum (\cite[Definition~6.2.3]{farah2019combinatorial}; see the discussion of the Downward L\"owenheim--Skolem Theorem, Theorem~7.1.9  of \cite{farah2019combinatorial} and  the entire Chapter 7 of this reference). 

\newpar{}One cannot prove that $\Sep(A)$ is $\sigma$-complete for every \cstar-algebra. For example, for the \cstar-algebra $A$ constructed from a Russell set in Proposition~\ref{P.Russell.example} (and many examples in \cite{blackadar2023hilbert}), $\Sep(A)$ is not $\sigma$-complete. However, the axiom AC$_\omega$ of countable choices implies that $\Sep(A)$ is $\sigma$-complete. 

Compare the following with Lemma~\ref{L.Axiomatizable}. \DC{} stands for the axiom of Dependent Choices.  

\newpar{}\begin{Proposition}\label{P.not-axiomatizable}
	While it is provable in \ZFC{} that having a tracial state is axiomatizable, this is not provable in  \ZF+\DC. 
\end{Proposition}

\begin{Proof}
The property of having a tracial state is  axiomatizable in continuous logic by  \cite[Theorem~2.5.1]{Muenster}. This proof uses an abstract characterization of axiomatizability via ultraproducts.

We will prove that in some model of \ZF{} there are elementarily equivalent \cstar-algebras $A$ and $B$ such that $B$ has a tracial state and $A$ does not. Once proved, this will imply that having a tracial state is not axiomatizable. 
As pointed out in \ref{P.ExistenceOfStates}, (see also \ref{P.ExistenceOfStates.1}) the commutative \cstar-algebra $A=\ell^\infty(\N)/c_0(\N)$ has no states (tracial or otherwise) if all sets of reals have the Property of Baire. A model with this property in which \DC{} holds was constructed starting from a model of \ZFC{} in \cite[Corollary~7.17]{shelah1984can}. 
 Since \DC{} implies \AC$_\omega$, in this model $\Sep(A) $ is $\sigma$-complete and therefore the downwards L\"owenheim--Skolem Theorem\footnote{We would like to thank Asaf Karagila for pointing out that \DC{} is equivalent to the downwards  L\"owenheim--Skolem theorem (Boolos, see \cite[Theorem~3.4]{aspero2021dependent}).}  (\cite[Theorem~7.1.9]{farah2019combinatorial}) implies that $A$ is elementarily equivalent to some of its separable \cstar-subalgebras. Every commutative, separable \cstar-algebra~$B$ is isomorphic to $C(X)$ for a compact metrizable space $X$ (Lemma~\ref{L.S(A)}) and in particular has tracial states. 
\end{Proof}

	 On the other hand, we have the following.

\newpar{}\label{T.absoluteness.C*}
\begin{Theorem} The statement `Every unital stably finite \cstar-algebra has a tracial state' is absolute between transitive models of \ZFC{} that contain all countable ordinals. 
\end{Theorem}

\begin{Proof} It suffices to prove (using AC) that if there is  a stably finite \cstar-algebra without a tracial state then there is one that is in addition separable. 
If $A$ is a unital \cstar-algebra, and every separable unital \cstar-subalgebra $B$ has a tracial state
$\tau_B$, extend $\tau_B$ to a state $\phi_B$ on $A$; then any cluster point of the net $(\phi_B)$
is a tracial state on $A$. Since being tracial is a closed condition, this cluster point is a tracial state. 
\end{Proof}

\newpar{}\label{Discussion.Absoluteness} Theorem~\ref{T.absoluteness.C*.separable} and Theorem~\ref{T.absoluteness.C*} imply that the standard set-theoretic methods such as the use of forcing or axioms  such  as the Continuum Hypothesis, or any other methods used in \cite{farah2019combinatorial},  cannot obtain independence of the statement from Theorem~\ref{T.absoluteness.C*}. This is because these methods result in transitive models  that contain all ordinals. 
However, an absolute statement may be independent from \ZF{} or even ZFC.  For example, G\"odel's sentence for \ZFC{} is absolute between transitive models of \ZF{} (or even between models of \ZF{} that contain no nonstandard integers), in spite of being independent from ZFC. The point here is that  (i)~its independence cannot be proved by standard set-theoretic methods such as passing to inner models or to forcing extensions and (ii) unless \ZFC{} is inconsistent, G\"{o}del's sentence is true. 
Since completeness of first-order logic implies that for every statement consistent with \ZFC{} there is a model of \ZFC{} in which this statement holds, this implies that for some statements all such models are ill-founded. Proving that such a concrete statement is independent from \ZFC{} would require methods so novel that such proof would be even more interesting from the point of view of set theory than its resolution would be from the point of view of the theory of \cstar-algebras. 

That said, there are a few examples of absolute, but independent, mathematical statements.  
In \cite{paris1977mathematical} a finitary consequence of Ramsey's theorem (and therefore of \ZFC) was isolated  that is expressible in the language of Peano arithmetic, but not provable in it (see also \cite{friedman1971higher}).  
For statements in analysis with similar properties  see \cite{foreman2020godel} and \cite{bausch2020undecidability}.

\bibliography{choiceless-hilbert}
\bibliographystyle{alpha}
\end{document}